\documentclass[12pt]{article}
\usepackage{paragrafos}
\usepackage[all,knot,poly]{xy}
\usepackage[utf8]{inputenc}
\usepackage{amsfonts}

\newtheorem{theorem}{Theorem}

\newtheorem{corollary}[theorem]{Corollary}

\newtheorem{definition}[theorem]{Definition}
\newtheorem{example}[theorem]{Example}

\newtheorem{proposition}[theorem]{Proposition}
\newtheorem{remark}[theorem]{Remark}

\newcommand{\qed}{\hfill \rule{0.5em}{0.5em}}
\newcommand{\qeda}{\rule{0.5em}{0.5em}}
\newenvironment{proof}[1][Proof]{\noindent\textbf{#1.} }{\hfill \rule{0.5em}{0.5em}}
\newenvironment{proof1}[1][Proof]{\noindent\textbf{#1.} }{}
\newcommand{\delete}[1]{}

\newcommand{\text}[1]{#1}

\begin{document}


\vspace*{-5cm}
\parbox[t]{\textwidth}
{
\title{Colourings and the Alexander Polynomial}
\author{Lu\'{\i}s Camacho\thanks{lcamacho@uma.pt}\\ SQIG - Instituto de Telecomunicações and \\ Department of Mathematics and Engineering,\\ University of Madeira\\ Funchal, Madeira, Portugal
\and F. Miguel Dion\'{\i}sio\thanks{fmd@math.ist.utl.pt}\\ SQIG - Instituto de Telecomunica\c{c}\~{o}es and \\  Department of Mathematics,\\ Instituto Superior Técnico \\ Lisbon, Portugal
\and Roger Picken\thanks{ rpicken@math.ist.utl.pt}\\ Center for Mathematical Analysis, \\ Geometry and Dynamical Systems and \\ Department of Mathematics,\\ Instituto Superior Técnico\\ Lisbon, Portugal}

\date{}
\maketitle
}
\begin{abstract}
In this paper we look for closed expressions to calculate the number of colourings of 
prime knots
 for  given linear Alexander quandles.
For this purpose the colouring matrices are simplified to a triangular form,
when possible. The operations used to perform this triangularization preserve
the property that the entries in each row add up to zero,  thereby simplifying the solution
of the equations giving the number of colourings. When the colouring matrices
(of prime knots up to ten crossings)  can be triangularized, closed expressions giving the 
number of colourings can be obtained in a straightforward way. 
We use these results to show that there are colouring matrices that cannot be 
triangularized. In the case of knots with triangularizable colouring matrices 
  we present a way to find linear Alexander quandles that distinguish
by colourings knots with different Alexander polynomials. The colourings of
knots with the same Alexander polynomial are also studied  as regards
when they can and cannot be distinguished by colourings.
\end{abstract}

\section{Introduction}
The number of quandle colourings of a knot diagram is a well known and rich invariant of a knot \cite{Carter:1245188}.
An interesting class of quandles are the linear Alexander quandles given by two coprime integers $n$ and $m$. 
A closely related invariant of the knot is its Alexander polynomial  \cite {alexander} and in this article we clarify a number of points about the precise relationship between the two invariants. Note that if the Alexander polynomial is replaced by the collection of all Alexander polynomials, then by a result of Inoue  \cite{Inoue2001a} these completely determine the number of quandle colourings for any Alexander quandle, linear or otherwise.

Although ingenious methods 
may be used to find the  number of quandle colourings of a knot diagram the simplest (and fastest) is to use a general formula
giving that number. In this article we find general expressions for the number of colourings 
which apply to all prime knots up to ten crossings with the exception of $12$ knots.  This  simply involves 
finding an analytical solution of the equation $AX=0$ for a simplified
 (triangularized)  matrix obtained from the  colouring matrix. The simplification uses standard transformations
 (multiplication of rows by units, adding rows and swapping rows or columns) 
 but we do not allow general operations on columns for the reason that such operations do not
 preserve the property that the entries in any row in the colouring matrix add up to zero. This property is useful
 since it facilitates the solution of the equations. The triangularized   
 matrices  were obtained using the Mathematica programming environment.   
 Using the same algorithms we also simplified  the matrices that we were unable to triangularize. 
 Furthermore we  show that some of these matrices cannot be triangularized.
The simplified matrices are related to the \textit{presentation matrices}  in Kawauchi's A Survey of Knot Theory  \cite{kawauchi1996survey}
since they are  obtained with similar, although less general,  types of operations.

The formulae for the number of colourings  allow us to draw some conclusions about properties 
of knots and their colourings using linear Alexander quandles: 
knots with triangularizable colouring matrices and different Alexander polynomials can always be 
distinguished by colourings and we conjecture that this is true if  the triangularizability  condition 
is dropped. On the other hand there are classes of  knots with the same Alexander polynomial
that cannot be distinguished by linear Alexander quandles and other classes with the same 
Alexander polynomial which can be distinguished by such colourings.

The structure of this article is as follows.
In section \ref{backg} we recall the basic notions of quandles
(\cite{Joyce198237}, \cite{matveev1982a}), in particular linear Alexander quandles, 
colourings of knot diagrams (\cite{pre06074274}, \cite{Carter:1245188}, \cite{Hayashi:1391930}, \cite{Inoue2001a}, \cite{Lopes2001a}, \cite{Nelson2002a}), 
the colouring matrix and the Alexander polynomial  
(\cite{alexander}, \cite{livingston1993knot}).

Section \ref{triangmat} presents the two types of triangularized matrices found when considering colouring matrices for
prime knots up to ten crossings and section \ref{typesofknots} compares these matrices with those  in 
\cite{kawauchi1996survey}. Section \ref{computing} presents the expressions giving the number of 
 colourings when the colouring matrices are triangularizable.
 In section \ref{comparing} we compare the number of colourings for knots with the same Alexander polynomial.  In section \ref{nontriang} we see that 
 it is useful to simplify colouring matrices even if they cannot be triangularized and in section \ref{nontriang2} we prove that there are colouring matrices that cannot be triangularized. In section \ref{alexandcol} we prove that knots with different Alexander polynomials and triangularizable colouring  matrices are distinguished by colourings. In section \ref{conc} we conclude and discuss further work. In the appendix we present the simplified  matrices, obtained from the colouring matrices, which are needed for the  expressions.

\section{Background}\label{backg}

In this section we recall the definition of a quandle and the notion of
coloring of a diagram. Since the knot quandle is a classifying invariant for 
knots (introduced independently by Joyce and Matveev - see \cite{Joyce198237} and \cite{matveev1982a}),  the number of colorings associated to a knot diagram is a knot invariant.
Later we present the notions of finite Alexander quandle, 
Alexander polynomial and   linear finite Alexander quandle.

\subsection{Quandles and colourings}
Colourings of the arcs of oriented knot diagrams with elements of a 
quandle generalize $mod\ p$ labellings of the arcs, that, in 
 turn, generalize the colorability invariant of R.\ Fox (with $p=3$ colours).
 They are also a generalization of arc labellings of oriented knot diagrams with group elements (see, for instance \cite{livingston1993knot}). 
 At each crossing the quandle elements labelling the arcs  are related by the quandle operation $\ast$. 
 The number of colourings is a knot invariant since different diagrams of the same knot have the same number of colourings 
 using a given quandle. Indeed, the definition of a quandle consists of precisely those properties of the binary operation $*$ that ensure 
  that colourings are preserved under the Reidemeister moves.

\begin{definition}
(Quandle) A quandle is a set X endowed with a binary operation, denoted $\ast$,
such that:

(a) for any $a\in X,a\ast a=a$

(b) for any $a$ and $b$ $\in X$, there is a unique $x\in X$ such that $%
a=x\ast b$

(c) for any $a,b$ and $c$ $\in X,\left( a\ast b\right) \ast c=\left( a\ast
c\right) \ast \left( b\ast c\right) $
\end{definition}

The definition of colouring of a knot diagram follows.

\begin{definition}
(colouring of a knot diagram) Let $X$ be a fixed finite quandle, $K$ a knot
(assumed to be oriented), $\stackrel{\rightarrow}{D}$ a diagram of $K$ and $R_{\stackrel{\rightarrow}{D}}$\ the set of arcs
of $\stackrel{\rightarrow}{D}$. A colouring of a diagram $\stackrel{\rightarrow}{D}$\ is a map $C:R_{\stackrel{\rightarrow}{D}}\longrightarrow X$
such that, at each crossing:

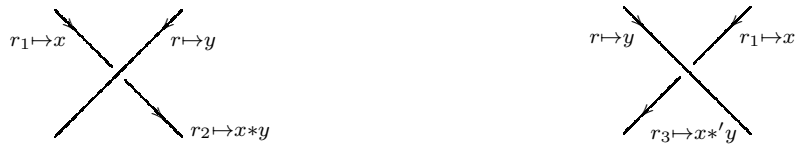
\parbox{6cm}{$$
\xy

\xygraph{
  !{0;/r4.0pc/:}
  !{\xoverv<>|>>><{r_1\mapsto x}|{r\mapsto y}>{\makebox[-2 cm]{$\labelstyle r_2\mapsto x\ast y$}}}
}

\endxy
$$
}
\hfill
\parbox{6cm}{
$$
\xy

\xygraph{
  !{0;/r4.0pc/:}
  !{\xunderv<>|>>>>{r_3\mapsto x\ast' y}|{r\mapsto y}<{r_1\mapsto x}}
}

\endxy
$$
}

i.e.\ if $C(r_{1})=x$ and $C(r)=y$, then $C(r_{2})=x\ast y$ for the crossing on the left, and for the crossing on the right,
 if $C(r_{1})=x$ and $C(r)=y$, then $C(r_{3})=x\ast' y$
where given $x,y$, $x\ast' y$ is the unique element such that $x=(x\ast'y)\ast y$.  \hfill \rule{0.5em}{0.5em}
\end{definition}

Colourings of knot diagrams using quandles are knot invariants in the following sense.

\begin{theorem}
Let $Q$ be a fixed finite quandle, $K$ a knot

and  $\stackrel{\rightarrow}{D}$ and $\stackrel{\rightarrow}{D'}$ oriented diagrams of $K$. 
Then the number of  colourings $C:R_{\stackrel{\rightarrow}{D}}\longrightarrow Q$ of diagram $\stackrel{\rightarrow}{D}$ using $Q$ is equal to the number of 
colourings  $C':R_{\stackrel{\rightarrow}{D'}}\longrightarrow X$ of diagram $\stackrel{\rightarrow}{D'}$ using $Q$.
\end{theorem}
For a more complete discussion of the results above and related topics
see \cite{Carter:1245188}, \cite{Inoue2001a},  \cite{kauffman1991knots},  \cite{Lopes2001a} and \cite{Nelson2002a}.

\subsection{Finite Alexander Quandles}\label{falexq}
Finite Alexander quandles have the form $\mathbb{Z}_{n}[t,t^{-1}]/h(t)$ where $n$ is an integer and $h(t)$ is a monic polynomial in $t$. These quandles have as elements equivalence classes of Laurent polynomials with coefficients in $\mathbb{Z}_{n}$,  where two polynomials are equivalent if their difference is divisible by $h(t)$.
The quandle operation is $a\ast b=t a+(1-t)b$. Note that this means equality of quandle elements, i.e.\ equivalence classes of Laurent polynomials. Recall that $c=a\ast'b$ is defined  to mean the same as $a=c\ast b$. From this it follows easily that $a\ast'b=t^{-1}a+(1-t^{-1})b$. 

For finite Alexander quandles the colouring condition at each crossing states that the label of the emerging arc is expressed as a linear combination of the labels of the other two arcs. Therefore one uses matrices to organize the colouring conditions (equations).

For that purpose it is important to have an enumeration of the arcs and an enumeration of the crossings.  Any enumeration will do. However, we will use for arcs an  enumeration that assigns $i+1$ to the emerging arc where $i$ is the number assigned to the under arc (see next figure), except for the last crossing when the emerging arc is already labelled (by $1$).  The enumeration for crossings is also arbitrary. We will use the enumeration suggested by the enumeration of arcs, i.e.\ the $k$-th crossing is the one with under arc also labelled $k$.\enlargethispage{\baselineskip}

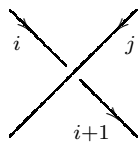
\parbox{4.5cm}{
$$
\xy

\xygraph{
  !{0;/r4.0pc/:}
  !{\xoverv<>|>>><{i}|{j}>{i+1}}
}
\endxy
$$
}
\hfill
\parbox{4.5cm}
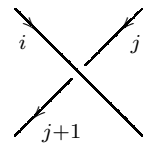
{
$$
\xy
\xygraph{
  !{0;/r4.0pc/:}
  !{\xunderv<>|>>><{j}|{i}>{j+1}}
}
\endxy
$$

}

Let $X_k$ be the label (in the quandle) of arc $k$. Then the colouring conditions can be written as follows:

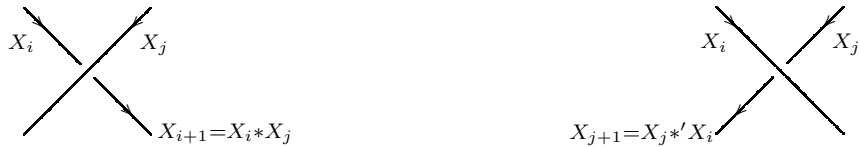
\parbox{4.5cm}{
$$
\xy

\xygraph{
  !{0;/r4.0pc/:}
  !{\xoverv<>|>>><{X_i}|{X_j}>{\makebox[-2.7 cm]{$\labelstyle X_{i+1}=X_i\ast X_j$}}}
}
\endxy
$$
}
\hfill
\parbox{4.5cm}{
$$
\xy
\xygraph{
  !{0;/r4.0pc/:}
  !{\xunderv<>|>>><{X_j}|{X_i}>{\makebox[-2.7 cm]{$\labelstyle  X_{j+1}=X_j\ast' X_i$}}}
}
\endxy
$$
}

The condition $X_{i+1}=X_i\ast X_j$ is $X_{i+1}=tX_i+(1-t)X_j$ or, equivalently,  $tX_i+(1-t)X_j-X_{i+1}=0$. 
The condition $X_{j+1}=X_j\ast' X_i$ is $X_{j+1}= t^{-1}X_j+(1-t^{-1})X_i$ 
and can be equivalently written $t^{-1}X_j+(1-t^{-1})X_i-X_{j+1}= 0$.

Given a (oriented) diagram $\stackrel{\rightarrow}{D}$ of a knot $K$, we can write the colouring conditions as a matrix equation 
$$AX=0$$

where $X$ is the vector of colouring unknowns ($X_1X_2\ldots X_i\ldots$ ) and each row in the matrix $A$ represents a colouring condition for one crossing in  $\stackrel{\rightarrow}{D}$. We will call the matrix $A$ a \textit{colouring matrix}. For example,

\[
\left[ 
\begin{array}{cccc}
t^{-1} & -1 & 0 & 1-t^{-1} \\ 
1-t & t & -1 & 0 \\ 
0 & 1-t^{-1} & t^{-1} & -1 \\ 
-1 & 0 & 1-t & t%
\end{array}
\right] \left[ 
\begin{array}{c}
X_{1} \\ 
X_{2} \\ 
X_{3} \\ 
X_{4}%
\end{array}
\right] =\left[ 
\begin{array}{c}
0 \\ 
0 \\ 
0 \\ 
0%
\end{array}
\right] 
\]

is the matrix equation corresponding to the following diagram of the knot $4_1$ (the figure-8 knot):
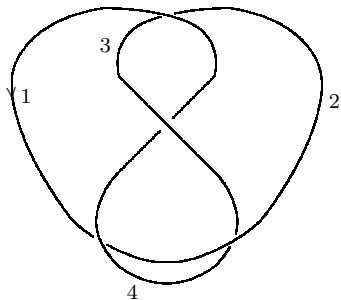
\vspace{-3\baselineskip}

$$\xygraph{ 
  !{0;/r1.5pc/:}
  !P9"e"{~:{(5,0):}~>{}}[u]
  !P5"d"{~:{(1.41421,0):}~>{}}[dd]
  !P4"a"{~:{(1.41421,0):}~>{}}[ddl]
  !P8"b"{~={45}~>{}} [rr]
  !P8"c"{~={45}~>{}} [ddl][u(0.1)]
  !P3"f"{~>{}}
  !{\xoverh~{"a2"}{"a1"}{"a3"}{"a4"}}
  !{\vover~{"b6"}{"b4"}{"f1"}{"b2"}}
  !{\vunder~{"c6"}{"c8"}{"f1"}{"c2"}}
  !{\vover~{"d3"}{"d1"}{"a2"}{"a1"}<{3}}
  !{\xcapv~{"c6"}{"f2"}{"b6"}{"f3"}<{4}}
  !{\hcap~{"d3"}{"e5"}{"b4"}{"e5"}<<<{1}}
  !{\hcap~{"d1"}{"e1"}{"c8"}{"e1"}<{2}}
}$$

\begin{remark}Obviously the number of colourings of a diagram in a 
linear Alexander quandle is the number of solutions of $AX=0$.
\end{remark}

\begin{remark}Instead of writing the condition $X_{j+1}=X_j\ast' X_i$ as $X_{j+1}= t^{-1}X_j+(1-t^{-1})X_i$ one could recall that $X_{j+1}=X_j\ast' X_i$ is equivalent to $X_j=X_{j+1}\ast X_i$ and this can be written as $X_j=tX_{j+1}+(1-t)X_i$ or as  $tX_{j+1}-X_j+(1-t)X_i=0$. This remark will be useful in the next section.
In this case the matrix defining the colouring conditions for $4_1$ is:

\[
\left[ 
\begin{array}{cccc}
-1 & t & 0 & 1-t\\ 
1-t & t & -1 & 0 \\ 
0 & 1-t & -1& t \\ 
-1 & 0 & 1-t & t%
\end{array}
\right] 
\]

\end{remark}

\subsection{The Alexander Polynomial}\label{alexpol}
The Alexander polynomial $\mbox{Alex}_K(t)$ (\cite{alexander}) of a knot $K$ is a knot invariant that is computed as follows (\cite{livingston1993knot}). First pick an oriented diagram $\stackrel{\rightarrow}{D}$ for $K$. Number the arcs of the diagram, and separately number the crossings. Next, define an $N\times N$ matrix, where $N$ is the number of crossings (and arcs) in the diagram, according to the following procedure: 

If the crossing numbered $l$ is right-handed with arc $i$ passing over arcs $j$ and $k$, as illustrated below by the diagram on the  left, enter $1-t$ in column $i$ of row $l$, enter  $-1$ in column $j$ of that row, and enter  $t$ in column $k$ of the same row. If the crossing is left-handed, as illustrated by the diagram on the right, enter  $1-t$ in column $i$ of row $l$,  $t$ in column $j$ and enter  $-1$ in column $k$ of row $l$. All of the remaining entries of row $l$ are $0$. (An exceptional case occurs if any two of $i$, $j$ or $k$ are equal. In this exceptional case the sum of the entries described above is put in the appropriate column).

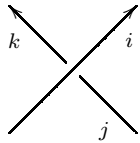
\parbox{4.5cm}{
$$
\xy
\xygraph{
  !{0;/r4.0pc/:}
  !{\xoverv=<<{k}|{i}>{j}>}
}
\endxy
$$
}
\hfill
\parbox{4.5cm}
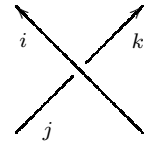
{
$$
\xy
\xygraph{
  !{0;/r4.0pc/:}
  !{\xunderv=<<{k}|{i}>{j}}
}
\endxy
$$
}

\begin{remark} We can think of the rules defining the rows in the matrix above as an implicit way of defining a condition on the colourings of the arcs. If, as before, $X_i$, $X_j$ and $X_k$ denote unknowns for colourings of the arcs $i, j$ and $k$ respectively, then the rule for row $l$ of a right-handed crossing  can be equivalently written as $(1-t)X_i-X_j+tX_k=0$. This is the same as $X_j=tX_k+(1-t)X_i$. This in turn is equivalent to $X_j=X_k\ast X_i$.
The rule for the left-handed crossing translates to $(1-t)X_i+tX_j-X_k=0$ and this is the same as $X_k=X_j\ast X_i$. If we write $X_j$ in terms of $X_k$ and $X_i$, we obtain $X_j=X_k\ast' X_i$. This can also be written as $X_j=t^{-1}X_k+(1-t^{-1})X_i$. 
The original equation (for the left-handed crossing) can be recovered by multiplying by $-t$. 
Note that this changes the determinant of $A$ by factors of $-t$. 
This does not change the Alexander polynomial (see definition \ref{alexmatrix} below) that is defined up to products of $\pm t$. 
 Note also that the exceptional cases are no longer ``exceptional'' - they consist of the special cases of the equations when some of the unknowns are equal. 

The conditions above are summarized in the following diagram.

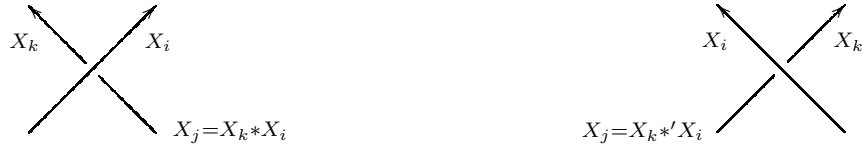
\parbox{4.5cm}{
$$
\xy
\xygraph{
  !{0;/r4.0pc/:}
  !{\xoverv=<<{X_k}|{X_i}>{\makebox[-2.7 cm]{$\labelstyle X_{j}=X_k\ast X_i$}}}
}
\endxy
$$
}
\hfill
\parbox{4.5cm}{
$$
\xy
\xygraph{
  !{0;/r4.0pc/:}
  !{\xunderv=<<{X_k}|{X_i}>{\makebox[-2.7 cm]{$\labelstyle  X_{j}=X_k\ast' X_i$}}}
}
\endxy
$$
}

Note that if we reverse all  arrows in the crossings we obtain precisely the colouring conditions presented before. For this reason the matrix $A$ encoding the colouring conditions for the oriented diagram $\stackrel{\rightarrow}{D}$ of knot $K$ is essentially the same as the matrix just defined but calculated for $\stackrel{\leftarrow}{D}$, i.\ e.\ the same diagram $D$ but with the orientation reversed (and keeping the labels of arcs and crossings the same).
\end{remark}

The definition of the Alexander polynomial follows.

\begin{definition} (from \cite{livingston1993knot})\label{alexmatrix}
The $(N-1)\times(N-1)$ matrix obtained by removing the final row and column from the $N\times N$ colouring matrix described above is called the Alexander matrix of $K$. The determinant of the Alexander matrix is called the Alexander polynomial of $K$. (The determinant of a $0\times 0$ matrix is defined to be $1$.)

Different diagrams of the same knot may lead to different Alexander polynomials related by a sign and a factor which is a power of $t$. This equivalence class is the Alexander polynomial. It is normal to normalize it (\cite{crowell1977introduction}) choosing a polynomial with ``no negative powers of $t$ and a positive constant term".
\end{definition}

\subsection{Linear Finite Alexander Quandles}\label{linFin}
In the following we will be dealing with linear finite Alexander quandles which
have the form $\mathbb{Z}_{n}[t,t^{-1}]\ /\ (t-m)$, where $n$ and $m$ are integers and $n, m$ are coprime.  Recall that the elements are equivalence classes of Laurent polynomials having the same remainder when divided by $t-m$. Obviously the polynomial $t$ is in the same equivalence class as the constant polynomial $m$,  since $t=(t-m)+m$.
Similarly $t^{-1}$ is equivalent to $m^{-1}$ (the inverse of $m$ in $\mathbb{Z}_{n}$), since 
$t^{-1}-m^{-1}=-m^{-1}t^{-1}(t-m)$.
It follows that any polynomial is equivalent to some number in $\mathbb{Z}_{n}$ and that one can identify $\mathbb{Z}_{n}[t,t^{-1}]\ /\ (t-m)$ with $\mathbb{Z}_{n}$.
The quandle operation can be written as  $a\ast b=m a+(1-m)b\ (\mbox{mod}\ n)$ and
$a\ast'b=m^{-1}a+(1-m^{-1})b\ (\mbox{mod}\ n)$.
Again for the knot $4_1$ the colourings in any linear finite Alexander quandle are equivalently given by the following equation:

\[
\left[ 
\begin{array}{cccc}
m^{-1} & -1 & 0 & 1-m^{-1} \\ 
1-m & m & -1 & 0 \\ 
0 & 1-m^{-1} & m^{-1} & -1 \\ 
-1 & 0 & 1-m & m%
\end{array}
\right] \left[ 
\begin{array}{c}
X_{1} \\ 
X_{2} \\ 
X_{3} \\ 
X_{4}%
\end{array}
\right] =\left[ 
\begin{array}{c}
0 \\ 
0 \\ 
0 \\ 
0%
\end{array}
\right] 
\]

where the colourings $X_i$ belong to $\mathbb{Z}_{n}$ and the equalities hold in $\mathbb{Z}_{n}$ (i.e.\ equality  mod\ $n$).

\section{Triangularized matrices}\label{triangmat}
For linear finite Alexander quandles $\mathbb{Z}_n[t,t^{-1}]/(t-m)$ the colourings are the solutions of $AX=0$ where $A$ is 
the colouring matrix and equality is mod $n$.  In order to find expressions for the number of colourings of a knot in different 
quandles it is convenient to rewrite the colouring matrix in triangular form.  For that purpose we have written algorithms   in
Mathematica  
that transform the
original matrix into an   equivalent one  by multiplying rows by $-1$, $m$ and $m^{-1}$ ($m$ is invertible since $\mbox{gcd}(m,n)=1$) and 
adding them, i.e.\ by replacing a row  by its sum with another one. Swapping rows\footnote{In fact swapping rows can be  achieved by suitably combining addition of rows and multiplication  by $-1$, $m$ and $m^{-1}$.} or columns is also allowed. 
The definition of equivalent matrices, in this sense, follows.

\begin{definition}\label{equivmatrices}Let $K$ be a knot and $A$ its colouring matrix. The matrix $B$ is equivalent to $A$ if it is obtained from $A$ by a sequence of matrices each obtained from the previous one by one of the following operations: a) multiplication of a row by $m$, or $m^{-1}$ or $-1$;  b) replacing a row by its sum with some row; c) swapping two rows and d) swapping two  columns.  If there is a triangular matrix equivalent to $A$ we say that $A$ is triangularizable.
\end {definition}

We will prove in section \ref{nontriang2} that there are, however,  matrices that cannot be triangularized  using these operations. 
Nevertheless  it was possible to triangularize in this way almost all colouring matrices for prime knots up to $10$ crossings, with the exception of the following $12$ knots:
$$9_{35}; 9_{38}; 9_{41}; 9_{47}; 9_{48}; 9_{49}; 10_{69}; 10_{101}; 10_{108}; 10_{115}; 10_{157}; 10_{160}$$
If we allow column operations analogous to the row operations of type a) and b) it is possible to triangularize $2$ of these namely the knots $9_{41}$ and $10_{108}$. See  the discussion in section \ref{typesofknots}. However, the expressions for the number of colourings presented in section \ref{computing}  depend on the property  that the entries in each row  in the colouring matrix add up to zero. This property is preserved under swapping columns and   row operations  of type a) and b) and c)  but not under more general column operations.

For the prime knots up to $10$ crossings with colouring matrices that we were able
to triangularize (i.e.\ all except the previously mentioned $12$) the final matrices had one of two different forms, that we will call type I and type II.
Both  have one final row of zeros. 
\enlargethispage{-3\baselineskip}

\begin{center}
type I%

\mbox{
$
\left[
\begin{array}
[c]{ccccc}%
1 & \lambda_{12} (m)& \cdots & \cdots &  \lambda_{1N} (m)\\
0 & \ddots & \ddots & \cdots & \vdots\\
\vdots & 0 & 1 &  \lambda_{N-2\ N-1}  (m)&  \lambda_{N-2\ N} (m)\\
\vdots & \vdots & 0 & \alpha(m)& -\alpha(m)\\
0 & 0 & \cdots & \cdots & 0
\end{array}
\right]
$}

\vspace{2\baselineskip}

Matrices of type I have one ``interesting row'', i.e.\  only one of the other rows does not have $1$ in the diagonal,
 and matrices of type II have two such rows:

type II%

\mbox{$\left[
\begin{array}
[c]{cccccc}%
1 &  \lambda_{12} (m) & \cdots & \cdots & \cdots &  \lambda_{1N} (m)\\
0 & \ddots & \ddots & \cdots & \cdots & \vdots\\
\vdots & 0 & 1 &  \lambda_{N-3\ N-2} (m)&  \lambda_{N-3\ N-1}  (m)&  \lambda_{N-3\ N} (m)\\
\vdots & \vdots & 0 & \alpha_{1}(m) & \beta_{1}(m) & -(\alpha_{1}(m)+\beta_{1}(m))\\
\vdots & \vdots & \vdots & 0 & \alpha_{2}(m) & -\alpha_{2}(m)\\
0 & 0 & 0 & \cdots & \cdots & 0
\end{array}
\right]  $}
\end{center}

Recall that the Alexander polynomial is the determinant of the Alexander matrix obtained from the colouring
 matrix by removing the final row and the final column. 
Therefore for $N\times N$ matrices of type I the Alexander polynomial is the polynomial 
$\alpha(m)$ at entry $N-1,N-1$ and for  $N\times N$ matrices of type II the Alexander polynomial
 is $\alpha_1(m)\times \alpha_2(m)$, the product of the 
polynomials at entries   $N-2,N-2$ and $N-1,N-1$.

In order to compute the number of colourings, only the ``interesting rows'' matter,
i.e.\ the rows excluding   the final  one that have an entry  not equal to $1$ in the diagonal.
For matrices of type I there is only one such row, the penultimate one.  Since the final column
is also redundant, matrices of type I are \textit{presented}  by   the polynomial 
$\alpha(m)$ at entry $N-1,N-1$, i.e.\ by their Alexander polynomial.  Furthermore, 
matrices of type II are \textit{presented} by two ``interesting'' rows 
(where the first $N-3$ columns and also the final column have been removed):

\begin{center}
\mbox{$\left[
\begin{array}
[c]{cc}%
 \alpha_{1}(m) & \beta_{1}(m) \\
 0 & \alpha_{2}(m) \\
\end{array}
\right]  $}
\end{center}

These considerations will be of use in section \ref{computing}.

\section{Types of knots}\label{typesofknots} 
We have seen before that the colouring matrices for prime knots up to ten crossings are either non-triangularizable or equivalent to 
a matrix of type I or of type II.
A comparison with the presentation matrices of \cite{kawauchi1996survey} makes sense since these are also obtained from colouring matrices 
by simplification operations including those in definition \ref{equivmatrices} as well as more general operations, in particular on columns.

We begin with the colouring matrices of knots that we were unable to triangularize.

\subsection{Knots with non-triangularized colouring matrices}
The colouring matrices   that we were unable to triangularize with row operations and swapping columns are the previously mentioned colouring matrices of the  $12$ following knots:
$$9_{35}; 9_{38}; 9_{41}; 9_{47}; 9_{48}; 9_{49}; 10_{69};10_{101}; 10_{108}; 10_{115};  10_{157}; 10_{160}$$

We will show  in section \ref{nontriang2} that it is impossible to triangularize 
(using row operations column swaps) the colouring matrices for  the $5$ knots $9_{35}, 9_{47}, 9_{48}, 9_{49}$ and $10_{157}$.  They also 
have non-triangular presentation matrices in  appendix F of A Survey of Knot Theory (\cite{kawauchi1996survey}). 
The $4$ knots $10_{69}, 10_{101}, 10_{115}$ and  $10_{160}$ 
also have non-triangular presentation matrices in that appendix. We conjecture that it is also impossible to triangularize
the colouring matrices for 
these  $4$ knots because they have non-factorizable Alexander polynomials (see section \ref{nontriang2}).

If we allow more general  column operations
it is possible to triangularize the colouring matrices for $2$ of the remaining $3$ knots, 
namely those of  the knots $9_{41}$ and $10_{108}$.  The  colouring matrix of  $9_{41}$ becomes type II and the one for $10_{108}$ 
 becomes type I in agreement with \cite{kawauchi1996survey}.  We have not yet been able to
triangularize the  colouring matrix for knot $9_{38}$ that has a triangularized presentation matrix (type II)
in appendix F of \cite{kawauchi1996survey}.

\subsection{Type II}
Using only row operations and swapping columns  we have obtained  type II colouring matrices for the following $21$ knots:
$$8_{18}; 9_{37}; 9_{40}; 9_{46}; 10_{61}; 10_{63}; 10_{65}; 10_{74}; 10_{75}; 10_{98}; 10_{99}; 10_{103}; $$
$$10_{106}; 10_{122}; 10_{123};10_{140};10_{142}; 10_{144}; 10_{147};10_{155}; 10_{164}$$

The  $18$ knots $8_{18}, 9_{37}, 9_{40}, 9_{46}, 10_{61}, 10_{63}, 10_{74}$,  $10_{75}$, $10_{98}, 10_{99}, 10_{103}, 
10_{122}$, $10_{123}$,
$10_{140}$, $10_{142}$, $10_{144},10_{155}$ and $10_{164}$ have type II matrices
and the presentation matrices for these knots in  A Survey of Knot Theory (\cite{kawauchi1996survey}) are also of type II.

 Furthermore
the  colouring matrices for   knots $10_{106}$ and $10_{147}$ can be further simplified if other operations on columns are allowed, therefore becoming type I, also
in agreement with  A Survey of Knot Theory.
One knot, $10_{65}$,  has a  type II colouring matrix  according to our algorithm and we were unable to simplify its colouring matrix even  when column operations were allowed. 
However,  it is type I according to  A Survey of Knot Theory (\cite{kawauchi1996survey}).

\subsection{Type I}
All $216$ type I knots\footnote{We use the expression ``type I knot'' to refer to a knot with a colouring matrix equivalent to a type I matrix. 
 The expression ``type II knot" has a similar meaning.} obtained by processing the colouring matrix allowing only row operations and swapping of columns are also type I knots according to  A Survey of Knot Theory (\cite{kawauchi1996survey}).  
These are listed in the appendix.

\subsection{Conclusion}
For knots with colouring matrices equivalent to  type I or type II matrices we will provide an explicit expression for the number of colourings using any linear Alexander quandle. 
Therefore, for this purpose it is  irrelevant if a type II  colouring matrix may be further simplified to type I 
since such an expression can be found in any case and all we want is this expression.
What is relevant  is whether the colouring matrix can be triangularized or not using only row operations and swapping of columns since  
the expressions  depend on the property  that the entries in each row  in the colouring matrix add up to zero.

There are $12$ colouring matrices  that we were unable to triangularize (using row operations and swapping of columns). 
Five (those of the knots $9_{35}, 9_{47}, 9_{48}, 9_{49}$ and $10_{157}$) are proven not to be  triangularizable by 
row operations and swapping of columns in section \ref{nontriang2}, and  we conjecture there 
that the same is true of another four (those of the knots $10_{69}, 10_{101}, 10_{115}$ and  $10_{160}$). 
We were also unable to triangularize  the colouring matrix of knot $9_{38}$. 
The remaining two colouring matrices are those of knots $9_{41}$ and  $10_{108}$ that we were only able to triangularize using other
column operations that do not preserve the property that the sum of entries in each row equals zero.
Therefore we were able to find a general expression for the colourings of all but the following $12$ knots:
  
  \begin{center}
  
  $9_{38}$ and $9_{35}, 9_{47}, 9_{48}, 9_{49}, 10_{157}; 10_{69}, 10_{101}, 10_{115}, 10_{160}$ and $9_{41}; 10_{108}$.

\end{center}

\section{Computing the number of colourings}\label{computing}
Recall that for linear finite Alexander quandles  $\mathbb{Z}_{n}[t,t^{-1}]\ /\ (t-m)$
the colourings are the solutions of $AX=0$ where equality is mod $n$. 
 In this section we are concerned with calculating the number of colourings assuming that an equivalent triangular 
matrix has been found.
Note that for a fixed $m$ and $n$ each polynomial entry $p$ yields a concrete value in $\mathbb{Z}_{n}$.

Although not necessary it is convenient to replace the values $v$ thus obtained
by their   equivalent $v\ \mbox{mod} \ n$. 

For example, the following matrix is triangular and a colouring matrix for knot $8_{18}$ (we omit some entries):

{\tiny
\[
\left[
\begin{array}
[c]{cccccccc}%
1 & \lambda_{12} (m)& \cdots & \cdots & \cdots & \cdots & \cdots & \lambda_{18}(m)\\
0 & 1 & \ddots & \cdots & \cdots & \cdots & \cdots & \vdots\\
0 & 0 & 1 & \ddots & \cdots & \cdots & \cdots & \vdots\\
0 & 0 & 0 & 1 & \ddots & \cdots & \cdots & \vdots\\
0 & 0 & 0 & 0 & 1 & \lambda_{56} (m)& \cdots & \lambda_{58}(m)\\
0 & 0 & 0 & 0 & 0 & -1+m-m^{2} & m-m^2+m^3 &1-2 m+2 m^2-m^3\\
0 & 0 & 0 & 0 & 0 & 0 & 1-4 m+5 m^2-4 m^3+m^4 &-1+4 m-5 m^2+4 m^3-m^4 \\
0 & 0 & 0 & 0 & 0 & 0 & 0 & 0
\end{array}
\right]
\]
}
 
For the quandle $\mathbb{Z}_{15}[t,t^{-1}]/(t-8)$ (choosing $m=8$ and $n=15$) 
we obtain the matrix:%

\[
A_{8_{18}}=\left[
\begin{array}
[c]{cccccccc}%
1 & \lambda_{12} (m) & \cdots & \cdots & \cdots & \cdots & \cdots & \lambda_{18} (m)\\
0 & 1 & \ddots & \cdots & \cdots & \cdots & \cdots & \vdots\\
0 & 0 & 1 & \ddots & \cdots & \cdots & \cdots & \vdots\\
0 & 0 & 0 & 1 & \ddots & \cdots & \cdots & \vdots\\
0 & 0 & 0 & 0 & 1 & \lambda_{56} (m) & \cdots & \lambda_{58} (m)\\
0 & 0 & 0 & 0 & 0 & 3 & 6 & 6\\
0 & 0 & 0 & 0 & 0 & 0 & 12 & 3\\
0 & 0 & 0 & 0 & 0 & 0 & 0 & 0
\end{array}
\right]
\]

The solutions of $A_{8_{18}}X=0\ \mbox{mod}\ 15$ are the colourings of knot $8_{18}$ 
using the quandle $\mathbb{Z}_{15}[t,t^{-1}]/(t-8)$. Note that adding the entries in each row  yields $0\ \mbox{mod}\ 15$.  
The number of solutions is easily determined (see below). 

A result that we will use often is the well known linear congruence theorem.

\begin{proposition}If $a$ and $b$ are integers and $n$ is a positive integer, then the congruence $ax=b\ \mbox{mod}\ n$ has a solution for $x$ if and only if $b$ is divisible by the greatest common divisor  of $a$ and $n$, $d=\mbox{gcd}(a,n)$. When this is the case, and $x_0$ is a solution of $ax=b\ \mbox{mod}\ n$, then the set of all solutions is given by $\{x_0+k\frac{n}{d}, k\in \mathbb{Z}\}$. In particular, there will be exactly  $d=\mbox{gcd}(a,n)$  solutions in the set $\{0,...,n-1\}$.
\end{proposition}

We proceed to find general expressions for the number of colourings for types I and II.  The number of solutions
for triangular matrices of any type could in principle  be determined using standard methods for solving systems of linear
congruences.
See remark \ref{generalsolution} where we comment on this briefly.

\subsection{Type I}
We now want to find the (number of) solutions of  $AX=0\ \mbox{mod}\ n$ where $A$ is of type I and $X=(X_1,...,X_N)$ are unknowns with values 
in the quandle $\mathbb{Z}_n[t,t^{-1}]/(t-m)$. Recall that the elements of the quandle are $\{0,1,...,n-1\}$. 
A brute force method would try all possible combinations of these values for $(X_1,...,X_N)$ and check for which choices all  equations hold.
For  matrices in triangular form we find values for each $X_i$ that are solutions of the equation in row $i$ and depend on previously determined $X_j$s, 
starting with the final row and proceeding upwards. 

\begin{figure}[h]
\begin{center}
\mbox{
$
\left[
\begin{array}
[c]{ccccc}%
1 & \lambda_{12} (m)& \cdots & \cdots &  \lambda_{1N} (m)\\
0 & \ddots & \ddots & \cdots & \vdots\\
\vdots & 0 & 1 &  \lambda_{N-2\ N-1}  (m)&  \lambda_{N-2\ N} (m)\\
\vdots & \vdots & 0 & \alpha(m)& -\alpha(m)\\
0 & 0 & \cdots & \cdots & 0
\end{array}
\right]
$}
\end{center}
\caption{Type I matrix}
\end{figure}

The expression for matrices of type I depends only on the final two rows of the matrix and 
is stated in the next proposition.

\begin{proposition}\label{propt1} Let $A$ be an $N\times N$ triangular  matrix of type I equivalent to the colouring matrix of a knot $K$. \label{coltypeI}
Then $C_Q(K)$, the number of colourings  of $K$ using the linear finite Alexander quandle $Q=\mathbb{Z}_n[T,T^{-1}]/(T-m)$
is $$C_Q(K)=n\times\mbox{gcd}(\mbox{Alex}(m),n).$$
\end{proposition}
\begin{proof} Recall that we find values for each $X_i$ that are solutions of the equation in row $i$ 
and depend on previously determined $X_j$s.

The equation corresponding to the last  row  is $0X_N=0\  \mbox{mod}\ n$ and holds always. Therefore any of
the $n$ values in $\{0,1,...,n-1\}$ is a solution for $X_N$. Now the equation corresponding to the row above this
is $\alpha(m)X_{N-1}-\alpha(m)X_N=0\  \mbox{mod}\ n$, or equivalently $\alpha(m)(X_{N-1}-X_N)=0\  \mbox{mod}\ n$.

Setting $Y_{N-1}=X_{N-1}-X_N$ the equation becomes  $\alpha(m)Y_{N-1}=0\  \mbox{mod}\ n$. 
From the linear congruence theorem there will be $d=\mbox{gcd}(\alpha(m),n)$ solutions for $Y_{N-1}$,
namely $Y_{N-1}= k\frac{n}{d}, k=0,...,d-1$.
 Since   $Y_{N-1}=X_{N-1}-X_N$, for each possible value 
of $X_N$ there will be  $d=\mbox{gcd}(\alpha(m),n)$ values for $X_{N-1}$, namely 
$X_{N-1}= X_N+k\frac{n}{d}, k=0,...,d-1$  such that  this equation holds.

The equation for  the $(N-2)th$ row  is $X_{N-2} +  \lambda_{N-2,N-1}X_{N-1} + \lambda_{N-2,N}X_{N}  =0\  \mbox{mod}\ n$ and 
has a unique solution for $X_{N-2}$,  namely  $X_{N-2} = - \lambda_{N-2,N-1}X_{N-1} -\lambda_{N-2,N}X_{N} \  \mbox{mod}\ n$. 
The other rows behave similarly. 
Therefore the total number of solutions is the number of solutions for $X_{N-1}$ and $X_N$, 
$n\times d$, where $d=\mbox{gcd}(\alpha(m),n)$. Recall from above that for matrices of type I,  $\alpha(m)$ is the
Alexander polynomial of the knot. Therefore, the number of colourings, i.e.\ the number of solutions of $AX=0\ \mbox{mod}\ n$   is 
$n\times\mbox{gcd}(\mbox{Alex}(m),n)$.
\end{proof}

An obvious consequence of proposition \ref{propt1} is that knots having type I matrices and the 
same Alexander polynomial cannot be distinguished by colourings of linear quandles.

\begin{corollary} \label{t1t1}The following pairs of knots cannot be distinguished by linear Alexander quandles since they have type I matrices and the same Alexander polynomial.
The final item consists of three knots that cannot be distinguished from each other by linear quandles.

 \begin{itemize}
 
 \item $5_{1},10_{132} \hfill (1-m+m^2-m^3+m^4)$

 \item $7_{4},9_{2}\hfill (4-7 m+4 m^2)$

 \item $7_{5},10_{130}\hfill (2-4 m+5 m^2-4 m^3+2 m^4)$

 \item $7_{6},10_{133}\hfill (1-5 m+7 m^2-5 m^3+m^4)$

\item $8_{3},10_{1}\hfill (4-9 m+4 m^2)$

\item $8_{5},10_{141}\hfill (1 - 3 m + 4 m^2 - 5 m^3 + 4 m^4 - 3 m^5 + m^6)$

 \item $8_{8},10_{129}\hfill (2-6 m+9 m^2-6 m^3+2 m^4)$

\item  $8_{10},10_{143}\hfill (1-3 m+6 m^2-7 m^3+6 m^4-3 m^5+m^6)$

 \item $8_{16},10_{156}\hfill (1-4 m+8 m^2-9 m^3+8 m^4-4 m^5+m^6)$
 
 \item $8_{21},10_{136}\hfill (1 - 4 m + 5 m^2 - 4 m^3 + m^4)$

 \item $9_{15},10_{166}\hfill (2-10 m+15 m^2-10 m^3+2 m^4)$

 \item $9_{20},10_{149}\hfill (1-5 m+9 m^2-11 m^3+9 m^4-5 m^5+m^6)$

 \item $9_{28},9_{29}\hfill (1-5 m+12 m^2-15 m^3+12 m^4-5 m^5+m^6)$

\item  $10_{10},10_{165}\hfill (3-11 m+17 m^2-11 m^3+3 m^4)$

 \item $10_{12},10_{54}\hfill (2-6 m+10 m^2-11 m^3+10 m^4-6 m^5+2 m^6)$

\item $10_{18},10_{24}\hfill (4-14 m+19 m^2-14 m^3+4 m^4)$

\item  $10_{20},10_{163}\hfill (3-9 m+11 m^2-9 m^3+3 m^4)$

\item  $10_{23},10_{52}\hfill (2-7 m+13 m^2-15 m^3+13 m^4-7 m^5+2 m^6)$

\item  $10_{25},10_{56}\hfill (2-8 m+14 m^2-17 m^3+14 m^4-8 m^5+2 m^6)$

 \item $10_{28},10_{37}\hfill (4-13 m+19 m^2-13 m^3+4 m^4)$

\item  $10_{31},10_{68}\hfill (4-14 m+21 m^2-14 m^3+4 m^4)$

\item $10_{34},10_{135}\hfill (3-9 m+13 m^2-9 m^3+3 m^4)$

\item $10_{127},10_{150}\hfill (1-4 m+6 m^2-7 m^3+6 m^4-4 m^5+m^6)$

 \item $8_{14},9_{8} , 10_{131}\hfill(2-8 m+11 m^2-8 m^3+2 m^4)$  \rule{0.5em}{0.5em}

 \end{itemize}
 \end{corollary}
 
 On the other hand type I knots $K_1$ and $K_2$ with different Alexander polynomials can always be distinguished by linear quandles. 
 One has to find $m$ and $n$ such that  $\mbox{gcd}(\mbox{Alex}_{K_1}(m),n)\neq\mbox{gcd}(\mbox{Alex}_{K_2}(m),n)$.
 This implies  $|\mbox{Alex}_{K_1}(m)|\neq|\mbox{Alex}_{K_2}(m)|$. We shall see that for triangularizable matrices such values always exist  in theorem \ref{alexdist}.
 The next example is a simple illustration of this fact.
 \begin{example} 
Consider knots $8_1$ and $8_2$. They both have type I matrices but different Alexander polynomials:
$\mbox{Alex}_{K_1}(m)=3-7m+3m^2$ is the Alexander polynomial of knot $8_1$ and 
$\mbox{Alex}_{K_2}(m)=1-3m+3m^2-3m^3+3m^4-3m^5+m^6$ is the Alexander polynomial of knot $8_2$.
They differ in absolute value for all $m\geq 3$ since $m=1$ is the only integer solution of 
$\mbox{Alex}_{K_1}(m)=\mbox{Alex}_{K_2}(m)$ and $m=2$ is the only integer solution of 
$\mbox{Alex}_{K_1}(m)=-\mbox{Alex}_{K_2}(m)$. For $m=3$ one has $\mbox{Alex}_{K_1}(3)=9$ and 
$\mbox{Alex}_{K_2}(3)=181$.  Choosing $n=181$ one has $\mbox{gcd}(9,181)=1$ and $\mbox{gcd}(181,181)=181$ 
and therefore the knots are distinguished in this quandle. Note that the choice of $n$ is not always obvious.
  \end{example}

\subsection{Type II}
When the matrix is of type II  the process of calculating the number of solutions is similar 
to that of type I for the two final rows. In this case a third equation has to be considered.

\begin{figure}[h]
\begin{center}
\mbox{$\left[
\begin{array}
[c]{cccccc}%
1 &  \lambda_{12}\hfill (m) & \cdots & \cdots & \cdots &  \lambda_{1N} (m)\\
0 & \ddots & \ddots & \cdots & \cdots & \vdots\\
\vdots & 0 & 1 &  \lambda_{N-3\ N-2} (m)&  \lambda_{N-3\ N-1}  (m)&  \lambda_{N-3\ N} (m)\\
\vdots & \vdots & 0 & \alpha_{1}(m) & \beta_{1}(m) & -(\alpha_{1}(m)+\beta_{1}(m))\\
\vdots & \vdots & \vdots & 0 & \alpha_{2}(m) & -\alpha_{2}(m)\\
0 & 0 & 0 & \cdots & \cdots & 0
\end{array}
\right]  $}
\end{center}
\caption{Type II matrix}
\end{figure}

\begin{proposition} Let $A$ be a triangular matrix   of type II equivalent to a colouring matrix of a knot $K$.
Then $C_Q(K)$, the number of colourings  of $K$ in the linear finite Alexander quandle $Q=\mathbb{Z}_n[T,T^{-1}]/(T-m)$
is $$C_Q(K)= n\times \mbox{gcd}(\alpha_2(m),n)\times \mbox{gcd}(\beta_1(m)\frac{n}{\mbox{gcd}(\alpha_2(m),n)},\mbox{gcd}(\alpha_1(m),n)).$$
\end{proposition}
\begin{proof}

Reasoning precisely as before in the proof for type I matrices the pairs  $(X_N,X_{N-1})$
 that are solutions of the two final equations $0X_N=0\ \mbox{mod}\  n$ and  $\alpha_2(m)Y_{N-1}=0\ \mbox{mod}\ n$,
  where $Y_{N-1}=X_{N-1}-X_N$ are those where $X_N\in \{0,1,...,n-1\}$ and for each such $X_N$, $X_{N-1}= X_N+k\frac{n}{d_2}, k=0,...,d_2-1$,
where  $d_2=\mbox{gcd}(\alpha_2(m),n)$.

The row above these yields the equation 
$\alpha_1(m)X_{N-2}+\beta_1(m)X_{N-1} -(\alpha_1(m)+\beta_1(m))X_N=0\ \mbox{mod}\ n$, that  can be rewritten as 
$\alpha_1(m)(X_{N-2}-X_{N})+\beta_1(m)(X_{N-1} -X_N)=0\ \mbox{mod}\ n$.
 Letting $Y_{N-2}=X_{N-2}-X_{N}$ and $Y_{N-1}=X_{N-1}-X_N$ (as before)
we obtain $$\alpha_1(m)Y_{N-2}+\beta_1(m)Y_{N-1}=0\ \mbox{mod}\ n$$

For each value of $X_N$ and $X_{N-1}$ we have from the two final rows as before that $Y_{N-1}=X_{N-1}-X_N=k\frac{n}{d_2}$, for $k=0,...,d_2-1$.
Substituting above one obtains $\alpha_1(m)Y_{N-2}+\beta_1(m)k\frac{n}{d_2}=0\ \mbox{mod}\ n$. 
This equation only has solutions for those values of $k$ such that $\beta_1(m)k\frac{n}{d_2}$ is divisible by 
$d_1=\mbox{gcd}(\alpha_1(m),n)$. And, if there is one solution then there will be $d_1$ solutions.

  We check how many of the values  $\beta_1(m)k\frac{n}{d_2}$ for $k=0,...,d_2-1$ are multiples of $d_1$. 
   This is equivalent to $\beta_1(m)\frac{n}{d_2}K=0\ \mbox{mod}\ d_1$ 
 that has $d_3=\mbox{gcd}(\beta_1(m)\frac{n}{d_2},d_1)$ solutions  
 and these are $K=t\times \frac{d_1}{d_3}$ for $t=0,...,d_3-1$.  
 Now we have to check which of these $K'$s are in  $0,...,d_2-1$,
 i.e.\ such that  $t\times \frac{d_1}{d_3}<d_2$. It is easy to check 
 that there are $c_3=d_2\times \frac{d_3}{d_1}$ possible values for $t$ namely $0,...,c_3-1$.
  Therefore there are $c_3$ values of $k$ such that $\beta_1(m)k\frac{n}{d_2}$ is a multiple of $d_1$.

Summing up we have that for each value of $X_N$ in $\{0,...,n-1\}$ there will be $d_2=\mbox{gcd}(\alpha_2(m),n)$ 
values for $X_{N-1}$ that are solutions of the final equations but only $c_3=d_2\times \frac{d_3}{d_1}$  of them
leads to solutions for $X_{N-2}$. Each of these, however, gives
  $d_1$ solutions for $X_{N-2}$. Therefore there are $n\times c_3 \times d_1$ solutions since the other $X_j$s with $j<N-2$ are 
  determined by the values of these three. Substituting one gets 
  $ C_Q(K)=n\times d_2\times \frac{d_3}{d_1}\times d_1= n\times d_2\times d_3=
  n\times \mbox{gcd}(\alpha_2(m),n)\times \mbox{gcd}(\beta_1(m)\frac{n}{d_2},d_1)=
    n\times \mbox{gcd}(\alpha_2(m),n)\times \mbox{gcd}(\beta_1(m)\frac{n}{\mbox{gcd}(\alpha_2(m),n)},\mbox{gcd}(\alpha_1(m),n))$.
\end {proof}

\begin{example}
Recall the previous example of  the colourings of the knot $8_{18}$ for  the
 quandle $\mathbb{Z}_{15}[t,t^{-1}]/(t-8)$ (where $m=8$ and $n=15$)  .

\[
A_{8_{18}}=\left[
\begin{array}
[c]{cccccccc}%
1 & \lambda_{12} & \cdots & \cdots & \cdots & \cdots & \cdots & \lambda_{18}\\
0 & 1 & \ddots & \cdots & \cdots & \cdots & \cdots & \vdots\\
0 & 0 & 1 & \ddots & \cdots & \cdots & \cdots & \vdots\\
0 & 0 & 0 & 1 & \ddots & \cdots & \cdots & \vdots\\
0 & 0 & 0 & 0 & 1 & \lambda_{56} & \cdots & \lambda_{58}\\
0 & 0 & 0 & 0 & 0 & 3 & 6 & 6\\
0 & 0 & 0 & 0 & 0 & 0 & 12 & 3\\
0 & 0 & 0 & 0 & 0 & 0 & 0 & 0
\end{array}
\right]
\]

The numbers relevant for computing the colourings are the following.

\begin{center}
\mbox{$\left[
\begin{array}
[c]{cc}%
 3& 6\\
 0 & 12\\
\end{array}
\right]  $}
\end{center}
The number of colourings is given by $$n\times \mbox{gcd}(\alpha_2(m),n)\times \mbox{gcd}(\beta_1(m)\frac{n}{\mbox{gcd}(\alpha_2(m),n)},\mbox{gcd}(\alpha_1(m),n)).$$ 
Substituting one obtains $15\times \mbox{gcd}(12,15)\times\mbox{gcd}(6\frac{15}{\mbox{gcd}(12,15)},\mbox{gcd}(3,15))=
15\times 3\times\mbox{gcd}(6\times 5,3)=15\times 3\times 3.$
The number of colourings is $135$.
\end{example}

More interesting is to use the number of colourings for type II matrices for the purpose of distinguishing knots. 
We will see in proposition \ref{aaaa} that knots with triangularized colouring matrices and different 
Alexander polynomials can  be distinguished by linear quandles. In particular example \ref{bbb}  uses  type II matrices.

\begin{remark}\label{generalsolution}
Note that we could extend these results to more general triangular matrices. In that case, however, it is difficult to find an 
expression for the solutions, and, therefore for the number of solutions. A general algorithm could work by 
finding the solutions for $X_i$ using previous solutions of $X_j$, with $j>i$. For each equation it is easy to 
check what values of the previous $X_j$'s give solutions (the independent term must be a multiple of the gcd of the entry on the diagonal and $n$). 
The solutions themselves, that are needed for rows above, can be computed with the extended Euclidean algorithm (available in some computer systems, 
such as Mathematica). We have written such a program and it agrees with the values for type I and II matrices.  However, for prime knots up to 10 crossings 
it is not relevant because such triangular matrices do not occur.
\end{remark}

\section{Comparing knots with the same Alexander polynomial}\label{comparing}
In this section we compare knots with the same Alexander polynomial.  Since
all prime knots (except   $9_{38}$) up to ten crossings with non-triangularizable colouring matrices have Alexander polynomials 
distinct from all other prime knots  up to ten crossings we consider only triangularizable
knots (for $9_{38}$ see below).

We will show in proposition \ref{aaaa}  that knots with triangularized colouring matrices and different 
Alexander polynomials can  be distinguished by linear quandles.  It is therefore interesting to understand if 
knots with the same Alexander polynomial but different types of matrices can be distinguished by colourings.

We first consider pairs of knots with triangularized type I or type II  matrices. Note that 
when the matrices are both type I and have the same Alexander polynomial  the knots cannot 
be distinguished by linear quandles (see corollary \ref{t1t1}).  Moreover, there are no examples of pairs of prime knots
up to ten crossings with the same Alexander polynomial and type II colouring matrices.
However,  if one counts  the colouring matrix of knot  $9_{38} $ -that we were not able to triangularize- as type II because it occurs as such in \cite{kawauchi1996survey}, then  both  $9_{38}$ and $10_{63}$ have type II matrices and the same Alexander polynomial. We were unable to find a linear Alexander quandle that distinguishes them.

We are now left with one  case:  when one of the matrices is type I and the other is type II.
We begin with pairs of such knots that can  be distinguished by linear Alexander quandles.
The first knot of the pair is type I and the second is type II. The values of $m$ and $n$ specify a linear Alexander quandle that distinguishes the knots.

\begin{proposition}
The following pairs of knots have the same Alexander polynomial but they can be distinguished by linear Alexander quandles.
The first knot of the pair is type I and the second is type II. The values of $m$ and $n$ specify a linear Alexander quandle that distinguishes the knots.
\begin{itemize}
\item $6_{1} - 9_{46}$  ($n = 3, m = 2$)
\item $8_{9} - 10_{155}$ ($ n = 5, m = 4$)
\item $9_{24} - 8_{18}$    ($n= 6, m = 5$)
\item $10_{40} - 10_{103}$ ($n= 5, m = 4$)
\item $10_{42} - 10_{75}$ ($n= 3, m = 2$)
\item $10_{59} - 9_{40}$ ($n= 5, m = 4$)
\item $10_{67} - 10_{74}$     ($n= 3, m = 2$)
\item $10_{87} - 10_{98}$  ($n= 3, m = 2$)
\end {itemize}
\end {proposition}

However, it is not always possible to distinguish by linear quandles knots that have the same Alexander quandle
and different types of matrices. The  knots  $8_{20}$ (type I) and $10_{140}$ (type II)  cannot
be distinguished by linear Alexander quandles as we see below. 
 
\begin{proposition}
Knots  $8_{20}$ (type I) and $10_{140}$ (type II) cannot be distinguished by linear Alexander quandles.
\end{proposition}
\begin{proof}Both knots have the same Alexander polynomial $A(m)=(1-m+m^2)^2=1 - 2 m   +3 m^2 - 2 m^3 + m^4$.
Since $8_{20}$
 is type I the number of colourings  in any linear Alexander quandle is given by $C_1(m,n)=n\times\mbox{gcd}(\mbox{A}(m),n)$.
The relevant entries for the type II matrix of $10_{140}$ are

\[
\left[ 
\begin{array}{ccc}
1-m+m^{2}  & -2m^{2} \\ 
 0 & 1-m+m^{2}  \\ 
\end{array}
\right] 
\]

Therefore the number of colourings of   $10_{140}$  in any linear Alexander quandle is given by 
$C_2(m,n)= n\times \mbox{gcd}(\alpha(m),n)\times \mbox{gcd}(\beta(m)\frac{n}{\mbox{gcd}(\alpha(m),n)},\mbox{gcd}(\alpha(m),n))$,
where $\alpha(m)=1-m+m^{2}$ and $\beta(m)=-2m^2$. 

We now show that $C_2(m,n)=C_1(m,n)$. Since $\alpha(m)=1-m+m^{2}=m(m-1)+1$ is odd and $m$ and $n$ 
are coprime then $\beta(m)=-2m^2$ is also coprime with $\mbox{gcd}(\alpha(m),n)$. Therefore
 $\mbox{gcd}(\beta(m)\frac{n}{\mbox{gcd}(\alpha(m),n)},\mbox{gcd}(\alpha(m),n))= 
 \mbox{gcd}(\frac{n}{\mbox{gcd}(\alpha(m),n)},\mbox{gcd}(\alpha(m),n))$.  Recall that $a\times\mbox{gcd}(b,c)=\mbox{gcd}(a\times b,a\times c)$
with $a\geq 1$  and note that $\mbox{gcd}(\mbox{gcd}(a,n)^2,n)= \mbox{gcd}(a^2,n)$. Indeed, 
$\mbox{gcd}(a^2,n)=\mbox{gcd}(\frac{a^2}{\mbox{gcd}(a,n)^2}\mbox{gcd}(a,n)^2,n)=\mbox{gcd}((\frac{a}{\mbox{gcd}(a,n)})^2\mbox{gcd}(a,n)^2,n)=\mbox{gcd}(\mbox{gcd}(a,n)^2,n)$ since 
$\frac{a}{\mbox{gcd}(a,n)}$ and $n$ are coprime. Therefore

$ \begin{array} {llr}
  C_2(m,n)&= n\times \mbox{gcd}(\alpha(m),n)\times \mbox{gcd}(\frac{n}{\mbox{gcd}(\alpha(m),n)},\mbox{gcd}(\alpha(m),n))&\\
 &= n\times  \mbox{gcd}(n,\mbox{gcd}(\alpha(m),n)^2)&\\
 &  = n\times  \mbox{gcd}(\alpha(m)^2,n)&\\
 & =n\times\mbox{gcd}(\mbox{A}(m),n)= C_1(m,n),&\\
 \end{array}$

  since $\alpha(m)^2$ is the Alexander polynomial $A(m)$.\end{proof}

The following pairs of knots, where the first one is type I and the second one is type II,  despite using a battery of 
thousands of linear quandles, could not be distinguished, thus we conjecture they cannot be distinguished by linear quandles.

\begin{itemize}
\item $10_{77} - 10_{65}$ 
\item $9_{28} - 10_{164}$
\item $9_{29} - 10_{164}$
\item$8_{11} -10_{147}$
\end{itemize}

We were unable to simplify the colouring matrix of  $10_{65}$  even using other column operations but it is type I in \cite{kawauchi1996survey}.
The first three pairs of knots can be distinguished by quadratic quandles (\cite{Dionisio2002a}) but we were unable to distinguish the pair 
$8_{11} -10_{147}$ even with non-linear Alexander quandles.
Note also that $10_{147}$ can be simplified to type I if other operations on columns are allowed.

\section{Simplified non-triangularized matrices}\label{nontriang} 
Even if a triangular form cannot be found, it is useful to ``optimize'' 
the colouring matrices by having as many $1$'s in the diagonal as possible. 
In that way the total number of colourings will depend only on the equations
of the other rows.  We were able to find equivalent matrices having at most $2\times 2$ relevant elements for the colouring matrices of 
all prime knots up to $10$ crossings. Consider for example the colouring matrix for the knot $9_{35}$:

\[
\left[
\begin{array}
[c]{lllllllll}%
\frac{1}{m} & -1 & 0 & 0 & 0 & 1-\frac{1}{m} & 0 & 0 & 0\\
0 & \frac{1}{m} & -1 & 0 & 0 & 0 & 0 & 1-\frac{1}{m} & 0\\
0 & 0 & \frac{1}{m} & -1 & 0 & 0 & 1-\frac{1}{m} & 0 & 0\\
0 & 0 & 0 & \frac{1}{m} & -1 & 0 & 0 & 0 & 1-\frac{1}{m}\\
0 & 1-\frac{1}{m} & 0 & 0 & \frac{1}{m} & -1 & 0 & 0 & 0\\
1-\frac{1}{m} & 0 & 0 & 0 & 0 & \frac{1}{m} & -1 & 0 & 0\\
0 & 0 & 1-\frac{1}{m} & 0 & 0 & 0 & \frac{1}{m} & -1 & 0\\
0 & 0 & 0 & 0 & 1-\frac{1}{m} & 0 & 0 & \frac{1}{m} & -1\\
-1 & 0 & 0 & 1-\frac{1}{m} & 0 & 0 & 0 & 0 & \frac{1}{m}%
\end{array}
\right]
\]

Our program  gave as output  the following simplified equivalent matrix\footnote{By still further processing using also operations on colums  we arrived at the same presentation as in in appendix F of  \cite{kawauchi1996survey}. }:

\[
\left[
\begin{array}
[c]{cccccc}%
1 & h_{1j} & \cdots & \cdots & \cdots & h_{1r}\\
0 & \ddots & \ddots & \cdots & \cdots & \vdots\\
\vdots & 0 & 1 & h_{r-3\ r-2} & h_{r-3\ r-1} & h_{r-3\ r}\\
\vdots & \vdots & 0 & 2-m & -1-m & -1+2m\\
\vdots & \vdots & \vdots & -3 & -2+7m & 5-7m\\
0 & 0 & 0 & 0 & 0 & 0
\end{array}
\right]
\]

Note that the relevant entries are the $4$ entries in the penultimate $2$ rows and $2$ columns:
\begin{center}
$\left(
\begin{array}{ccc}
 2-m &   & -1-m \\
 -3 &   & -2+7 m
\end{array}
\right)$
\end {center}
In the appendix only the relevant entries are shown for the $12$ non-triangularized matrices.

From the previous matrix we can compute the number of colourings considering only solutions for the $3$ final rows. 
A method of computing them would be simply by trying all $n^3$ possible combinations of values
and check those that are solutions of the final $3$ rows. More intelligent methods can be devised. One simple case that will be useful below is when 
 $m=2$ and $n=3$. In this case we have the $3$ final rows consisting only of zeros\footnote{The entries are $-3=0, 12=0, 9=0$ and $3=0$ mod $3$.}
    and therefore  $n^3=27$ colourings. 

\[
\left[
\begin{array}
[c]{cccccc}%
1 & h_{1j} & \cdots & \cdots & \cdots & h_{1r}\\
0 & \ddots & \ddots & \cdots & \cdots & \vdots\\
\vdots & 0 & 1 & h_{r-3\ r-2} & h_{r-3\ r-1} & h_{r-3\ r}\\
\vdots & \vdots & 0 & 0 & 0 & 0\\
\vdots & \vdots & \vdots & 0 & 0 & 0\\
0 & 0 & 0 & 0 & 0 & 0
\end{array}
\right]
\]

\section{Non-triangularizability results}\label{nontriang2}
It is important to note that there are colouring matrices 
(such as the one for the knot $9_{35}$) 
that do not have an equivalent triangular form, i.e.\ using operations on rows and swapping of columns 
it is not possible to transform the original matrix into a triangular one.
We show this in this section. 

The main reason why there are matrices that do not have an equivalent triangular form  is the following.
The product of the diagonal entries of a matrix in triangular form must be the Alexander polynomial 
(up to products of powers of $\pm m$). 
However some Alexander polynomials do not factorize. 
Therefore, in such cases if it is possible to triangularize the matrix, then  the Alexander polynomial occurs
in one of the entries of the diagonal, with all other entries except the final one being $1$. 
Such a matrix is equivalent 
to a type I  matrix and we know already how to determine its number of colourings for any quandle.
 We use this fact to show  that knots  $9_{35}$, $9_{47}$, $9_{48}$,  $9_{49}$ and $10_{157}$  have colouring matrices that 
 cannot be triangularized.
Indeed, we find quandles where the number of colourings of these knots does not coincide with the number 
given by the expression obtained when assuming that their colouring matrices can be triangularized.

Curiously,  a similar line of thinking should apply to  knots $10_{69}$, $10_{101}$, $10_{115}$ and $10_{160}$ 
since they also have Alexander polynomials that do not factorize. However, we were unable to find linear quandles 
that could serve as counter examples. We comment on this at the end of the section.

We now look for Alexander polynomials that are not factorizable, i.e.\ cannot be written as products of other polynomials in a non-trivial way.

\begin{definition}
An Alexander polynomial $\mbox{Alex}(m)$
with integer coefficients is said to be properly factorizable if the corresponding normalized form is properly factorizable. 
A polynomial (in non-negative powers of $m$) is properly factorizable if  it can be written as the product of  integer polynomials different from $\pm 1$ and itself.
\end{definition}

Note that for example $4+2m^2=2(2+m^2)$ is properly factorizable. The following example illustrates the fact that polynomials with non-integer roots may be properly factorizable.
\begin{example} The polynomial $8-2m^2-m^4$ has roots $m=2i$,  $m=-2i$, $m=\sqrt{2}$ and $m=-\sqrt{2}$ none of  which are integer. However,
 it can be written as the product of integer polynomials since $8-2m^2-m^4=-(m-2i)(m+2i)(m-\sqrt{2})(m+\sqrt{2})= -(m^2+4)(m^2-2)$. Therefore it is properly factorizable.
\end{example}

We now show that there are Alexander polynomials that are not properly factorizable.
\begin{proposition} The Alexander polynomials for knots $9_{35}$, $9_{47}$, $  9_{48}$, $  9_{49}$, $ 10_{69}$, $ 10_{101}$, $ 10_{115}$, 
$ 10_{157}$ and $ 10_{160}$ are not properly factorizable.

\begin{proof1}
The normalized Alexander polynomial of knot  
 $9_{35}$ is $\mbox{Alex}_{9_{35}}(m)=7 - 13 m  +7 m^2$. 
 Its roots are $r_1=\frac{1}{14} \left(13-3 i \sqrt{3}\right)$ and 
 $r_2=\frac{1}{14} \left(13+3 i \sqrt{3}\right)$.  
 The original polynomial is $7\times (m-r_1) \times(m-r_2)$ and it is not possible to 
 multiply these factors to yield integer polynomials other than $\mbox{Alex}_{9_{35}}(m)$.
 Therefore $\mbox{Alex}_{9_{35}}(m)=7 - 13 m  +7 m^2$ is not properly factorizable.
 
 The proof for the other knots is similar. We found their roots using Mathematica and tested the 
 finite products of  $(m-r_i)$  multiplied by the coefficient of the highest   degree monomial of $m$.
  Only the products  including all the $(m-r_i)$ factors yielded an integer polynomial  (the original polynomial)
 and all
 other products did not yield integer polynomials.\footnote{ Substituting  $m$ with values $0$ and $1$  non-integer values were obtained.}
 Here we list only the Alexander polynomials and omit the calculations. 

\begin{itemize}
\item $9_{47}\hfill (1 - 4 m + 6 m^2 - 5 m^3 + 6 m^4 - 4 m^5 + m^6)$
\item $ 9_{48}\hfill (1 - 7 m + 11 m^2 - 7 m^3 + m^4)$
\item $ 9_{49}\hfill (3 - 6 m + 7 m^2 - 6 m^3 + 3 m^4)$
\item $10_{69}\hfill (1 - 7 m + 21 m^2 - 29 m^3 + 21 m^4 - 7 m^5 + m^6)$
\item $ 10_{101}\hfill (7 - 21 m + 29 m^2 - 21 m^3 + 7 m^4)$
\item $ 10_{115}\hfill (1 - 9 m + 26 m^2 - 37 m^3 + 26 m^4 - 9 m^5 + m^6)$
\item $ 10_{157}\hfill (1 - 6 m + 11 m^2 - 13 m^3 + 11 m^4 - 6 m^5 + m^6)$  
\item $ 10_{160}\hfill (1 - 4 m + 4 m^2 - 3 m^3 + 4 m^4 - 4 m^5 + m^6)$ \qeda
\end{itemize}
\end{proof1}
\end{proposition}

We now show that in the case of  knots with  Alexander polynomials that are not factorizable, if their colouring matrices are
triangularizable then they are equivalent to a type I matrix.

\begin{proposition}Let $K$ be a knot  and $\mbox{Alex}_{K}(m)$ its Alexander polynomial. 
Assume that  $\mbox{Alex}_{K}(m)$ is not properly factorizable
and that there is a triangular colouring matrix $A$  for $K$.  
Then it is possible to transform $A$ into an equivalent type I matrix $B$. 

\begin{proof}
We show firstly that it is possible to transform $A$ in such a way that 
all entries  on the diagonal (except the final one that is $0$) are either $1$ or $\mbox{Alex}_{K}(m)$.
Recall that the entries in each row of the colouring matrix add up to zero. 
Since the operations that transformed the original colouring matrix into the
triangular matrix $A$ preserve this property we know that the entries
in each row of $A$ add up to zero. Since $A$ is assumed to be
triangular its final row must consist only of zeros. This row will be left unchanged.
Let $N$ be the number of rows (and columns) of $A$. Each remaining diagonal entry is a polynomial in $m$ and $m^{-1}$  and the product of the 
entries on the diagonal except for the final row is $\Pi_{i=1}^{N-1}a_{ii}=\pm m^k\mbox{Alex}_{K}(m)$, the Alexander polynomial possibly multiplied by  $\pm m^k$. Since  $\mbox{Alex}_{K}(m)$ is not properly factorizable, none of  the diagonal entries can be a proper factor and they are either $a_{ii}=\pm m^{k_i}$ or $a_{ii}=\pm m^{k_i}\mbox{Alex}_{K}(m)$ (this case can occur only once). Now multiply each row by $\pm m^{-k_i}$.

Next swap the column where $\mbox{Alex}_{K}(m)$ occurs with the penultimate column and then the row where $\mbox{Alex}_{K}(m)$ occurs with the penultimate row. The non-zero entries under the diagonal are turned into zero by repeatedly adding to the corresponding row suitable multiples of the rows above where $1$ occurs in the diagonal. Therefore the final matrix is type I. 
\end{proof}
\end{proposition}

Now we show that some colouring matrices cannot be triangularized.

\begin{proposition}The colouring matrices of knots $9_{35}$, $9_{47}$, $  9_{48}$, $  9_{49}$ and
$ 10_{157}$  cannot be triangularized.

\begin{proof1} Since these knots have non-properly factorizable Alexander polynomials we know that if they have a triangular colouring matrix 
then the number of colourings in any linear Alexander quandle $\mathbb{Z}_{n}[t,t^{-1}]\ /\ (t-m)$ must be $n\times\mbox{gcd}(\mbox{Alex}_K(m),n)$.
The proof consists now in exhibiting for each such knot $K$ a quandle where the true number of colourings is not equal to $n\times\mbox{gcd}(\mbox{Alex}_K(m),n)$.

\item $9_{35}$. We have seen in section \ref{nontriang} that for $m=2$ and $n=3$ the number of colourings of the knot $9_{35}$ is $27$. Its Alexander polynomial is 
$7-13m+7m^{2}$ and is not properly factorizable. The number of colourings computed using an equivalent triangular colouring matrix
would, however, be  $n\times\mbox{gcd}(\mbox{Alex}_K(m),n)=3\times\mbox{gcd}(7-26+28,3)=3\times\mbox{gcd}(9,3)=3\times 3=9$.

\item $9_{47}$.  We look for colourings using the quandle given by $m=2, n=3$.  For $m=2$ the Alexander polynomial $1 - 4 m + 6 m^2 - 5 m^3 + 6 m^4 - 4 m^5 + m^6$ is $9$. Thus the number of colourings computed from a triangular matrix for  $9_{47} $ would be $3\times \mbox{gcd}(9,3)=3\times 3=9$. However there are $3\times 3\times 3$ colourings as we can see from the relevant entries of the colouring matrix for this knot (after simplification).
The second matrix below is obtained from the first by putting $m=2$ and the third by substituting the mod $3$ values of the entries.

 \hspace{\fill}
$\left[
\begin{array}
[c]{cc}%
-1+4 m-m^2 &-2-m-m^2+m^3\\
 2-7 m & 3+4 m+2 m^2-m^4\\
\end{array}
\right]  $ \hspace{\fill}
$\left[
\begin{array}
[c]{cc}%
3 & 0\\
 -12 &3  \\
\end{array}
\right]$\hspace{\fill}
$\left[
\begin{array}
[c]{cc}%
0 & 0\\
 0 &  0\\
\end{array}
\right]$ \hspace{\fill}

\item $ 9_{48}$.  We look for colourings using the quandle given by $m=2, n=3$. For $m=2$ the Alexander polynomial $1 - 7 m + 11 m^2 - 7 m^3 + m^4$ is $-9$. Thus the number of colourings computed from a triangular matrix for  $9_{48} $ would be   $3\times \mbox{gcd}(-9,3)=3\times 3=9$.  However there are $3\times 3\times 3$ colourings, proceeding as in the previous case:

\hspace{\fill}$\left[
\begin{array}{ccc}
 2-m &   & 2-8 m+7 m^2-m^3 \\
 3 &   & 3-10 m+2 m^2+5 m^3-m^4
\end{array}
\right]$ \hspace{\fill}
$\left[
\begin{array}{ccc}
0 &   & 6 \\
 3 &   & 15
\end{array}
\right]$ \hspace{\fill}
$\left[
\begin{array}{ccc}
0 &   & 0 \\
0 &   & 0
\end{array}\right]$ \hspace{\fill}

\item $ 9_{49}$. We look for colourings using the quandle given by  $m=4, n=5$ since for $m=2, n=3$  this case is inconclusive. The Alexander polynomial $3 - 6 m + 7 m^2 - 6 m^3 + 3 m^4$ for $m=4$ is $475$. Thus the number of colourings computed from a triangular matrix for  $9_{49} $ would be   $5\times \mbox{gcd}(475,5)=5\times 5=25$.  However there are $5\times 5\times 5$ colourings.

 \hspace{\fill}
$\left[
\begin{array}
[c]{cc}%
-2 - m + m^2& 3 - m - m^2 - 2 m^3\\
3 - 2 m& -3 + 3 m + m^2\\
\end{array}
\right]  $ \hspace{\fill}
$\left[
\begin{array}
[c]{cc}%
10 & -145\\
 -5 &25  \\
\end{array}
\right]$\hspace{\fill}
$\left[
\begin{array}
[c]{cc}%
0 & 0\\
 0&  0\\
\end{array}
\right]$ \hspace{\fill}

\item $ 10_{157}$.  For $m=2, n=3$ this case is also inconclusive.
But for $m=6, n=7$ the Alexander polynomial  $1 - 6 m + 11 m^2 - 13 m^3 + 11 m^4 - 6 m^5 + m^6$ is $11809$. Then the number of colourings computed from a triangular matrix for  $10_{157} $ would be   $7\times \mbox{gcd}(11809,7)=7\times 7=49$.  However there are $7\times 7\times 7$ colourings.

 \hspace{\fill}
$\left[
\begin{array}
[c]{cc}%
4 - 3 m& -7 + 12 m - 9 m^2 + 6 m^3 - m^4\\
-1 + m^2& 2 - 3 m + m^2 - m^3\\
\end{array}
\right]  $ \hspace{\fill}
$\left[
\begin{array}
[c]{cc}%
-14 & -259\\
 35&-196 \\
\end{array}
\right]$\hspace{\fill}
$\left[
\begin{array}
[c]{cc}%
0& 0\\
 0 &  0\\
\end{array}
\right]$ \hspace{\fill}
\qed
\end{proof1}
\end{proposition}

We have already mentioned that  knots $10_{69}$, $10_{101}$, $10_{115}$ and $10_{160}$ also have Alexander polynomials that do not factorize.
It is remarkable that they do not have type I colouring matrices according to our calculations and according to \cite{kawauchi1996survey} but somehow behave like type I. Indeed the number of colourings using thousands of linear quandles coincides with the number obtained using the expression for type I matrices.
Presumably this means that a general formula for the number of colourings reduces to the type I formula in these cases because of some specific feature.

\section{Colourings and the Alexander polynomial}\label{alexandcol}
In this section we show that given two knots with triangularizable colouring matrices but different Alexander polynomials, a  linear finite  Alexander quandle can be exhibited that distinguishes the two knots by colourings. Note that this holds for any two knots which give triangularized matrices (they do not have to be type I or type II, they can be  \textit{any type}).

Firstly we note that there is an upper bound for the number of colourings of a triangular colouring matrix in any quandle.

\begin{proposition}
Let $K$ be a knot and  $A$  an $N\times N$ triangularized colouring matrix for $K$, with diagonal entries $a_{ii} (m)$, $i=1,...,N-1$.
Then the number of colourings of $K$ in  any  linear  finite Alexander quandle $Q=\mathbb{Z}_n[T,T^{-1}]/(T-m)$ satisfies
 $C_Q(K)\leq n\times\Pi_{i=1}^{N-1}\mbox{gcd}(a_{ii}(m),n)$.
 \end{proposition}
 \begin{proof}
 Trivial: We look for solutions upwards, starting with $X_N$. There will be $n$ solutions $X_N=0,\ldots,n-1$ which are possible values for $X_N$. 
 Fix one such value $X_N=v$. We now see for  $X_N=v$  how many solutions there are for the other variables.  
 The penultimate equation is $a_{N-1N-1}X_{N-1}- a_{N-1N-1}v=0 \ \mbox{mod}\ n$, since $a_{N-1N}=-a_{N-1N-1}$.
There are $\mbox{gcd}(a_{N-1N-1},n)$ solutions of this equation. Therefore we
 have $n\times \mbox{gcd}(a_{N-1N-1},n)$ solutions for $X_N$ and $X_{N-1}$. Now for each pair of values for $X_N=v_1$ and $X_{N-1}=v_2$ the equation  in row $N-2$ becomes $a_{N-2N-2}X_{N-2}+a_{N-2N-1}v_2+a_{N-2N}v_1=0\ \mbox{mod}\ n$. If each of these equations admits any solution then the number of solutions will be $\mbox{gcd}(a_{N-2N-2},n)$. Therefore there are at most $n\times \mbox{gcd}(a_{N-1N-1},n)\times \mbox{gcd}(a_{N-2N-2},n)$ solutions for  $X_N$, $X_{N-1}$ and $X_{N-2}$. The considerations for the other rows are similar.
 \end{proof}
 
 We now show that when $n$ is a multiple of all the $a_{ii}(m)$ the inequality becomes an equality.
 
 \begin{proposition}\label{prop:maxcol}
 Let $K$ be a knot and  $A$  a $N\times N$ triangular colouring matrix for $K$, 
 with diagonal entries $a_{ii} (m)$, $i=1,...,N-1$.
Given coprime $1<m<n$ such that  $a_{ii}(m)\neq 0$, for all of  $i=1,...,N-1$,  and furthermore such that $n$ 
is a multiple of all the $a_{ii}(m)$,  $i=1,...,N-1$,
the number of colourings of $K$ in  the  linear  finite Alexander quandle $Q=\mathbb{Z}_n[T,T^{-1}]/(T-m)$ is 
 $C_Q(K)= n\times\Pi_{i=1}^{N-1}\mbox{gcd}(a_{ii}(m),n)=n\times|\mbox{Alex}_K(m)|$.
 \end{proposition}
 \begin{proof}
 First it is convenient to rewrite the equations $AX=0\ \mbox{mod}\ n$ in terms of the variables $Y_i=X_i-X_N$, for $i=1,...,N-1$. 
 The final row remains unchanged but every other  row $a_{ii}(m)X_i+...+a_{iN-1}(m)X_{N-1}+a_{iN}(m)X_N$ 
 can be rewritten as  $a_{ii}(m)Y_i+...+a_{iN-1}(m)Y_{N-1}(m)$ since the entries  add up to zero and therefore 
 $a_{iN}=-(a_{ii}+...+a_{iN-1})$. 
 
 In the proof we use the fact that $\mbox{gcd}(a_{ii}(m),n)=|a_{ii}(m)|$ since 
$n$ is a multiple of each $a_{ii}(m)\neq 0$. That $n$ is a multiple of each $a_{ii}(m)$ can be written as
  $n=c\times\Pi_{i=1}^{N-1} |a_{ii}(m)|$.
We now show by induction that, given solutions for the rows below the $i^{th}$ row,
 there are $\mbox{gcd}(a_{ii}(m),n)=|a_{ii}(m)|$ solutions for row $i$. It is convenient to write $i=N-j$ and use $j$ for induction.
 We show that there are $\mbox{gcd}(a_{N-jN-j}(m),n)$ solutions for  $Y_{N-j}$ 
 and also that  each solution is a multiple of $\Pi_{k=1}^{N-j-1} a_{kk}(m)$,
 the product of the diagonal entries in the rows above the  $i^{th}$  row. 
 
 For $j=1$ we have $ \mbox{gcd}(a_{N-1N-1}(m),n)$ solutions of 
 $Y_{N-1}$ (recall the proof of proposition \ref{coltypeI}) .  
 The solutions are $k\times \frac{n}{\mbox{gcd}(a_{N-1N-1}(m),n)}=k\times \frac{c\times\Pi_{i=1}^{N-1} |a_{ii}(m)|}{|a_{N-1N-1}(m)|}=
 k\times c\times\Pi_{i=1}^{N-2} |a_{ii}(m)|=
 k_1\times\Pi_{i=1}^{N-2} a_{ii}(m)$.

 Now for $j>1$ the equation is  $a_{N-jN-j}(m)Y_{N-j}+a_{N-jN-j+1}(m)Y_{N-j+1}+...+a_{N-jN-1}(m)Y_{N-1}=0\ \mbox{mod}\ n$. 
Choosing a solution for each $Y_{N-j+1},...,Y_{N-1}$ by the induction hypothesis we obtain 
$a_{N-jN-j}(m)Y_{N-j} $ $+$  $a_{N-jN-j+1}(m)\times k_{j-1}\Pi_{i=1}^{N-j} a_{ii}(m) $ $+ \dots +a_{N-jN-1}(m)\times k_1\Pi_{i=1}^{N-2} a_{ii}(m) = 0\ \mbox{mod}\ n$.
Since all  terms except the first  include the factor $\Pi_{i=1}^{N-j} a_{ii}(m)$ we can rearrange this equation as $a_{N-jN-j}(m)Y_{N-j}+\beta(m)\times\Pi_{i=1}^{N-j} a_{ii}(m)=0\ \mbox{mod}\ n$. 
Now  we have that $\mbox{gcd}(a_{N-jN-j}(m),n)=|a_{N-jN-j}(m)|$ and the independent term $\beta(m)\times\Pi_{i=1}^{N-j} a_{ii}(m)$ is divisible by
$a_{N-jN-j}(m)$  so there are $\mbox{gcd}(a_{N-jN-j}(m),n)=|a_{N-jN-j}(m)|$ solutions, as follows from the linear congruence theorem.
Moreover, $a_{N-jN-j}(m)Y_{N-j}+\beta(m)\times\Pi_{i=1}^{N-j} a_{ii}(m)=0\ \mbox{mod}\ n$ can be equivalently rewritten as 
$$a_{N-jN-j}(m)\left(Y_{N-j} + \beta(m)\times\Pi_{i=1}^{N-j-1} a_{ii}(m)\right) = 0\ \mbox{mod}\ n$$ One solution of this equation is 
$Y_{N-j}=x_0=-\beta(m)\times\Pi_{i=1}^{N-j-1} a_{ii}(m)$. Any other solution is of the form $x_0 + k' \frac{n}{\mbox{gcd}(a_{N-jN-j}(m),n)}$. Now it is easy to check that both $x_0$ and $k' \frac{c\times\Pi_{i=1}^{N-1} a_{ii}(m)}{|a_{N-jN-j}(m)|}$ are multiples of $\Pi_{i=1}^{N-j-1} a_{ii}(m)$ and therefore so is each solution of the equation. This ends the proof by induction. The final result follows easily: there are $n$ solutions for $X_N$ and each solution for $Y_i=X_i-X_N$ yields a solution for $X_i$. Therefore there are $n\times\Pi_{i=1}^{N-1} |a_{ii}(m)|=n\times |\mbox{Alex}_K(m)|$ solutions.
 \end{proof}
 
 The main result of this section follows.

 \begin{proposition}\label{alexdist}\label{aaaa} 
Let $K_1$ and $K_2$ be knots with different Alexander polynomials 
$\mbox{Alex}_{K_1}(m)\neq \mbox{Alex}_{K_2}(m)$. 
Assume furthermore that their colouring matrices are triangularizable.
Then there is a  linear Alexander quandle that distinguishes them.
 \end{proposition}
 \begin{proof}
 Since the Alexander polynomials are different there will be an infinite number of values of $m$ such that 
 $|\mbox{Alex}_{K_1}(m)|\neq |\mbox{Alex}_{K_2}(m)|$. 
 Let $A$ be an $N_1\times N_1$ triangular matrix for $K_1$ and $B$ be an $N_2\times N_2$ triangular matrix for $K_2$.
There are also an infinite number of values of $m$ such that additionally 
 $a_{ii}(m)\neq 0$, where  $a_{ii}(m),  i=1,...,N_1-1$ are all but the final  diagonal entries of $A$ and also 
 such that $b_{ii}(m)\neq 0$, where  $b_{ii}(m), i=1,...,N_2-1$  are all but the final diagonal entries of $B$.
 Such values of $m$ are those that are not solutions of any of the equations  $a_{ii}(m)=0, i=1,...,N_1-1$,  $b_{ii}(m)=0, i=1,...,N_2-1$, or $\mbox{Alex}_{K_1}(m)=\mbox{Alex}_{K_2}(m)$ or  $\mbox{Alex}_{K_1}(m)=-\mbox{Alex}_{K_2}(m)$.
 We need a final condition on $m$: that $m$ is coprime with all $a_{ii}(m),  i=1,...,N_1-1$ and all $b_{ii}(m), i=1,...,N_2-1$.
 To find such an $m$ multiply (if needed) the diagonal entries by a power of $m$ so that they become polynomials without 
 negative powers of $m$ and a
non-zero constant term. An $m$ that is coprime with the constant term of the normalized $p(m)$ is coprime
 with $p(m)$.  
  There will also be an infinite number of such values of $m$, for example the prime numbers that are coprime with the constant terms of
 $a_{ii}(m),  i=1,...,N_1-1$ and  $b_{ii}(m), i=1,...,N_2-1$. This $m$ is coprime with 
 $\Pi_{i=1}^{N_1-1}|a_{ii}(m)|\times \Pi_{i=1}^{N_2-1}|b_{ii}(m)|=|\mbox{Alex}_{K_1}(m)|\times |\mbox{Alex}_{K_2}(m)|$. 
 Let $M$ denote  $\Pi_{i=1}^{N_1-1}|a_{ii}(m)|\times \Pi_{i=1}^{N_2-1}|b_{ii}(m)|=|\mbox{Alex}_{K_1}(m)|\times |\mbox{Alex}_{K_2}(m)|$.

 Choose $n$ to be a multiple of $M$ bigger that $m$ and coprime with $m$. If $M>m$ then choose $n=M$. 
 Otherwise multiply $M$ by the first prime bigger than $m$. We now have coprime $1<m<n$ and $n$
satisfying  the conditions of proposition \ref{prop:maxcol} for both knots $K_1$ and $K_2$.
 Therefore the number of colourings of $K_1$ and $K_2$ 
 in  the  linear finite Alexander quandle $Q=\mathbb{Z}_n[T,T^{-1}]/(T-m)$ satisfies

\hfill $C_Q(K_1)= n\times |\mbox{Alex}_{K_1}(m)|\neq n\times |\mbox{Alex}_{K_2}(m)|= C_Q(K_2).$
 \end{proof}
 
 \vspace{\baselineskip}

 We now illustrate the previous result with some examples.
 
 \begin{example} Firstly we consider distinguishing 
  two  type I knots $K_1$ and $K_2$ with different Alexander polynomials. 
  Their triangularized matrices have one final row of zeros,
  the Alexander polynomial in the penultimate diagonal entry 
  and $1$'s in all other diagonal entries.  We 
  have to find an $m$ such that  
  a) $|\mbox{Alex}_{K_1}(m)|\neq |\mbox{Alex}_{K_2}(m)|$;
  b) all diagonal entries except the last are non-zero and c) $m$ is coprime with all diagonal entries.

  In this case this simplifies to  a) $|\mbox{Alex}_{K_1}(m)|\neq |\mbox{Alex}_{K_2}(m)|$;
  b) $\mbox{Alex}_{K_1}(m)\neq 0$ and $\mbox{Alex}_{K_2}(m)\neq 0$ and
  c) $m$ is coprime with  $\mbox{Alex}_{K_1}(m)$ and $\mbox{Alex}_{K_2}(m)$.
 
Take  for example knots $K_1=3_1$ with Alexander polynomial $1-m+m^2$ and 
$K_2=4_1$ with Alexander polynomial $1-3 m+m^2$. Solving $1-m+m^2=1-3 m+m^2$ and
$1-m+m^2=-(1-3 m+m^2)$ we obtain  $m=0$ and $m=1$. The equations
$1-m+m^2=0$ and $1-3 m+m^2=0$ have no integer solutions. So any value of $m>1$ 
satisfies a) and b). Condition c) is fulfilled for $m$ coprime with the constant term of 
each Alexander polynomials that happens to be $1$ in both cases. Therefore any $m>1$ will do.
Choose $m=2$. Now $|\mbox{Alex}_{K_1}(2)|\times |\mbox{Alex}_{K_2}(2)|=3\times 1=3$.
Since $3$ is coprime with $2$ we can choose $n=3$ and the quandle with $m=2$, $n=3$  distinguishes the two  knots. Indeed,
the number of colourings of  $3_1$  is $3\times\mbox{gcd}(3,3)=3\times 3=9$ and the number of colourings of 
 $4_1$ is $3\times\mbox{gcd}(-1,3)=3\times 1=3$.
 
 \end{example}
 \begin{example}
 
 Consider now knots $K_1=10_{137}$ (type I) with Alexander polynomial $1 - 6 m + 11 m^2 - 6 m^3 + m^4$ 
 and $K_2=10_{155}$ (type II) with Alexander polynomial   $1 - 3 m + 5 m^2 - 7 m^3 + 5 m^4 - 3 m^5 + m^6$.
 The relevant entries of the triangularized colouring matrix for $10_{155}$ are the following:
 
\begin{center}
$\left(
\begin{array}{ccc}
 -1+2 m-m^2+m^3 &   & 0 \\
 0 &   & -1+m-2 m^2+m^3
\end{array}
\right)
$
 \end{center}
 
The equation $1 - 6 m + 11 m^2 - 6 m^3 + m^4=1 - 3 m + 5 m^2 - 7 m^3 + 5 m^4 - 3 m^5 + m^6$ has two integer solutions, $m=-1$ and $m=0$, the equation  $1 - 6 m + 11 m^2 - 6 m^3 + m^4=-(1 - 3 m + 5 m^2 - 7 m^3 + 5 m^4 - 3 m^5 + m^6)$ has one integer solution $m=1$,
and the equations $1 - 6 m + 11 m^2 - 6 m^3 + m^4=0$, $ -1+2 m-m^2+m^3=0$ and $ -1+m-2 m^2+m^3=0$ have no integer solutions.
We can choose  $m=2$ because it is coprime with the constant term of $1 - 6 m + 11 m^2 - 6 m^3 + m^4$, 
$ -1+2 m-m^2+m^3$ and $ -1+m-2 m^2+m^3$.  Thus we can choose $n=7$
 since $|\mbox{Alex}_{K_1}(2)|\times |\mbox{Alex}_{K_2}(2)|=1\times 7=7$.

We now confirm that this quandle distinguishes the knots. The number of colourings of knot $10_{137}$  for this 
quandle is $C_Q(K_1)=n\times\mbox{gcd}(\mbox{Alex}_{K_1}(2),n)= 7\times\mbox{gcd}(1,7)=7$.  
To determine the number of colourings of knot $10_{155}$ note that $\alpha_1(m)= -1+2 m-m^2+m^3$,
$\beta_1(m)= 0$ and $\alpha_2(m)= -1+m-2 m^2+m^3$. For  $m=2$,  $\alpha_1(2)=7$ and $\alpha_2(2)=1$.
Therefore $C_Q(K_2)= 7\times \mbox{gcd}(1,7)\times \mbox{gcd}(0,\mbox{gcd}(7,7))=7\times 1\times \mbox{gcd}( 7,7)=7\times 7=49$.
\end{example}
 \begin{example}\label {bbb}
 Consider now knots $K_1=8_{18}$ (type II) with Alexander polynomial $1 - 5 m + 10 m^2 - 13 m^3 + 10 m^4 - 5 m^5 + m^6$ 
 and $K_2=9_{37}$ (type II) with Alexander polynomial   $2 - 11 m + 19 m^2 - 11 m^3 + 2 m^4$. Their triangularized colouring matrices are the following: 
  
  \begin{center}
$8_{18}\text{ : }\left(
\begin{array}{ccc}
 -1+m-m^2 &   & m-m^2+m^3 \\
 0 &   & 1-4 m+5 m^2-4 m^3+m^4
\end{array}
\right)$
\end{center}

\begin{center}
$9_{37}\text{ : }\left(
\begin{array}{ccc}
 1-2 m &   & m+m^2 \\
 0 &   & -2+7 m-5 m^2+m^3
\end{array}
\right)$
\end{center}

The equation  $|\mbox{Alex}_{K_1}(m)|=|\mbox{Alex}_{K_2}(m)|$ has $m=-1$ and $m=1$ as its only integer roots. The equations
$ -1+m-m^2=0$, $1-4 m+5 m^2-4 m^3+m^4=0$ and $ 1-2 m =0$ have no integer roots. The equation $-2+7 m-5 m^2+m^3=0$ has $m=2$ as only integer root. Therefore we may choose $m=3$ which moreover is coprime with both constant terms in the diagonal entries. The product
$|\mbox{Alex}_{K_1}(m)|\times|\mbox{Alex}_{K_2}(m)|$ yields $49\times 5=245$ so we can choose $n=245$.

To determine the number of colourings of knot $8_{18}$ note that $\alpha_1(m)=  -1+m-m^2$,
$\beta_1(m)= m-m^2+m^3 $ and $\alpha_2(m)= 1-4 m+5 m^2-4 m^3+m^4$. 
For  $m=3$,  $\alpha_1(3)=-7$, $\beta_1(3)=21$ and $\alpha_2(3)=7$.
$C_Q(K_1)= 245\times \mbox{gcd}(7,245)\times \mbox{gcd}(21\frac{245}{\mbox{gcd}(7,245)},\mbox{gcd}(7,245))=
245\times 7\times \mbox{gcd}(21\times35,7)=245\times 7\times 7=12005$.

Now for knot $9_{37}$, $\alpha_1(m)=  1-2 m$,
$\beta_1(m)=m+m^2 $ and $\alpha_2(m)= -2+7 m-5 m^2+m^3$. 
For  $m=3$,  $\alpha_1(3)=-5$, $\beta_1(3)=12$ and $\alpha_2(3)=1$.
$C_Q(K_2)= 245\times \mbox{gcd}(1,245)\times \mbox{gcd}(12\frac{245}{\mbox{gcd}(1,245)},\mbox{gcd}(-5,245))=
245\times 1\times \mbox{gcd}(12\frac{245}{1},5)=245\times 1\times 5=1225.$
\end{example}
 
 \section{Acknowledgements} We thank Pedro Lopes from the  Department of Mathematics of IST  for ideas that led to this research;  
 Eduardo Marques de Sá from the Department of Mathematics of the University of Coimbra 
 for help with results on the diagonalization of matrices and J.\ Scott Carter from the Department of Mathematics and Statistics
of the
University of South Alabama for useful discussions in Lisbon.

 \section{Conclusions and Further Work}\label{conc}
We have presented general expressions for the number of colourings of prime knots using linear Alexander quandles when the colouring matrices can be triangularized into one of two forms.

These expressions allow us to conclude that knots with different Alexander polynomials
(and colouring matrices that have an equivalent triangularized form) 
are distinguishable by colourings. We have obtained such a triangular form for all but $12$ knots with up to ten crossings. In $5$  exceptional cases we prove that no triangular form exists. We were also able to make statements about knots with the same Alexander 
polynomial concerning when the number of colourings distinguishes or does not distinguish the knots. 

We conjecture that the condition above on the triangularizability of the colouring matrices can be dropped 
and that knots with different Alexander polynomials can always be distinguished by colourings. 
 Moreover, we expect that similar methods can be applied to knots having more that ten crossings. 
Indeed, the simplification algorithms apply to any knots. For knots with more than ten crossings that have 
equivalent type I or type II matrices we already have a general formula. A natural direction for future work is
to try and find general expressions for the number of solutions when the simplified colouring matrix is of a more general triangular or non-triangular type.
This may also help to elucidate why we were unable to prove that $4$ more knots with non-properly  factorizable Alexander polynomials have non-triangularizable colouring matrices.

\appendix 
\section{Non-triangularized matrices}
We were unable to triangularize the  colouring matrices for the following   12 knots where we display the relevant entries in the penultimate two rows and columns as in section \ref{nontriang}.
Note that the colouring matrices of 
$9_{41}$ and $10_{108}$ can be triangularized if more general column operations are allowed yielding type II and type I matrices respectively.
  \vspace{2\baselineskip}

\noindent\(9_{35}\text{ : }\left(
\begin{array}{ccc}
 2-m &   & -1-m \\
 -3 &   & -2+7 m
\end{array}
\right)\)

\noindent\(9_{38}\text{ : }\left(
\begin{array}{ccc}
 -1+m+m^2 &   & 4-4 m \\
 -5+7 m &   & 15-19 m+5 m^2
\end{array}
\right)\)

\noindent\(9_{41}\text{ : }\left(
\begin{array}{ccc}
 -1+m^2 &   & 4 m-3 m^2 \\
 -4+3 m &   & 13 m-12 m^2+3 m^3
\end{array}
\right)\)

\noindent\(9_{47}\text{ : }\left(
\begin{array}{ccc}
 -1+4 m-m^2 &   & -2-m-m^2+m^3 \\
 2-7 m &   & 3+4 m+2 m^2-m^4
\end{array}
\right)\)

\noindent\(9_{48}\text{ : }\left(
\begin{array}{ccc}
 2-m &   & 2-8 m+7 m^2-m^3 \\
 3 &   & 3-10 m+2 m^2+5 m^3-m^4
\end{array}
\right)\)

\noindent\(9_{49}\text{ : }\left(
\begin{array}{ccc}
 -2-m+m^2 &   & 3-m-m^2-2 m^3 \\
 3-2 m &   & -3+3 m+m^2
\end{array}
\right)\)

\noindent\(10_{69}\text{ : }\left(
\begin{array}{ccc}
 m-m^2-m^3 &   & 1-6 m+10 m^2-2 m^3 \\
 1-2 m+2 m^2 &   & -1+2 m-4 m^2+m^3
\end{array}
\right)\)

\noindent\(10_{101}\text{ : }\left(
\begin{array}{ccc}
 3-5 m+3 m^2 &   & -3+11 m-15 m^2+7 m^3 \\
 1-m+m^3 &   & -1+3 m-5 m^2+2 m^3
\end{array}
\right)\)

{\small
\noindent\(10_{108}\text{ : }\left(\begin{array}{ccc}
 -3 m-m^2 & \hspace{-1cm}   -3+8 m-10 m^2+12 m^3-10 m^4+6 m^5-2 m^6 \\
 \hspace{-0.5cm}  -11 m &   \hspace{-0.5cm}  -11+33 m-47 m^2+57 m^3-51 m^4+34 m^5-14 m^6+2 m^7
\end{array}\right)\)
}

\noindent\(10_{115}\text{ : }\left(
\begin{array}{ccc}
 1-m+m^2 &   & -3+3 m-m^2 \\
 2 m &   & 1-14 m+17 m^2-8 m^3+m^4
\end{array}
\right)\)

\noindent\(10_{157}\text{ : }\left(
\begin{array}{ccc}
 4-3 m &   & -7+12 m-9 m^2+6 m^3-m^4 \\
 -1+m^2 &   & 2-3 m+m^2-m^3
\end{array}
\right)\)

\noindent\(10_{160}\text{ : }\left(
\begin{array}{ccc}
 -3 m &   & 1+m+3 m^2-3 m^3-2 m^4+m^5 \\
 -2 m+m^2 &   & 1-m+3 m^2-4 m^3+m^4
\end{array}
\right)\)

\section{Type II colouring matrices}
In this section we list the relevant entries of the 21 type II  matrices obtained from colouring matrices using  row operations and swapping of columns.
Note that the colouring matrices of knots $10_{106}$ and $10_{147}$
can be simplified (become type I) if more general column operations are allowed.
The relevant entries of type II matrices are :
\begin{center}
\mbox{$\left[
\begin{array}
[c]{cc}%
 \alpha_{1}(m) & \beta_{1}(m) \\
 0 & \alpha_{2}(m) \\
\end{array}
\right]  $}
\end{center}
Given  a linear Alexander quandle
  $Q=\mathbb{Z}_{n}[t,t^{-1}]\ /\ (t-m)$ 
the number of colourings  is
   $$C_Q(K)= n\times \mbox{gcd}(\alpha_2(m),n)\times \mbox{gcd}(\beta_1(m)\frac{n}{\mbox{gcd}(\alpha_2(m),n)},\mbox{gcd}(\alpha_1(m),n)).$$
  
  \vspace{2\baselineskip}

\noindent\(8_{18}\text{ : }\left(
\begin{array}{ccc}
 -1+m-m^2 &   & m-m^2+m^3 \\
 0 &   & 1-4 m+5 m^2-4 m^3+m^4
\end{array}
\right)\)

\noindent\(9_{37}\text{ : }\left(
\begin{array}{ccc}
 1-2 m &   & m+m^2 \\
 0 &   & -2+7 m-5 m^2+m^3
\end{array}
\right)\)

\noindent\(9_{40}\text{ : }\left(
\begin{array}{ccc}
 1-4 m+5 m^2-4 m^3+m^4 &   & 0 \\
 0 &   & -1+3 m-m^2
\end{array}
\right)\)

\noindent\(9_{46}\text{ : }\left(
\begin{array}{ccc}
 2-m &   & -3 \\
 0 &   & 1-2 m
\end{array}
\right)\)

\noindent\(10_{61}\text{ : }\left(
\begin{array}{ccc}
 -1+m-m^2 &   & 1+m-m^2 \\
 0 &   & 2-3 m+m^2-3 m^3+2 m^4
\end{array}
\right)\)

\noindent\(10_{63}\text{ : }\left(
\begin{array}{ccc}
 -1+m-m^2 &   & 2 m^2 \\
 0 &   & -5+9 m-5 m^2
\end{array}
\right)\)

\noindent\(10_{65}\text{ : }\left(
\begin{array}{ccc}
 -1+m-m^2 &   & -1+m+m^2 \\
 0 &   & 2-5 m+7 m^2-5 m^3+2 m^4
\end{array}
\right)\)

\noindent\(10_{74}\text{ : }\left(
\begin{array}{ccc}
 -1+2 m &   & 0 \\
 0 &   & -4+8 m-7 m^2+2 m^3
\end{array}
\right)\)

\noindent\(10_{75}\text{ : }\left(
\begin{array}{ccc}
 1-4 m+3 m^2-m^3 &   & -1+2 m \\
 0 &   & 1-3 m+4 m^2-m^3
\end{array}
\right)\)

\noindent\(10_{98}\text{ : }\left(
\begin{array}{ccc}
 -2+3 m-3 m^2+m^3 &   & 1-m+m^2 \\
 0 &   & -1+3 m-3 m^2+2 m^3
\end{array}
\right)\)

\noindent\(10_{99}\text{ : }\left(
\begin{array}{ccc}
 1-2 m+3 m^2-2 m^3+m^4 &   & 0 \\
 0 &   & 1-2 m+3 m^2-2 m^3+m^4
\end{array}
\right)\)

\noindent\(10_{103}\text{ : }\left(
\begin{array}{ccc}
 -1+2 m-2 m^2 &   & -1+2 m-m^2+m^3 \\
 0 &   & 2-4 m+5 m^2-3 m^3+m^4
\end{array}
\right)\)

\noindent\(10_{106}\text{ : }\left(
\begin{array}{ccc}
 1-m+2 m^2-m^3 &   & -m+2 m^2-2 m^3+m^4 \\
 0 &   & -1+3 m-4 m^2+4 m^3-2 m^4+m^5
\end{array}
\right)\)

\noindent\(10_{122}\text{ : }\left(
\begin{array}{ccc}
 1-4 m+5 m^2-4 m^3+m^4 &   & -1+3 m-m^2 \\
 0 &   & -2+3 m-2 m^2
\end{array}
\right)\)

\noindent\(10_{123}\text{ : }\left(
\begin{array}{ccc}
 1-3 m+3 m^2-3 m^3+m^4 &   & 0 \\
 0 &   & 1-3 m+3 m^2-3 m^3+m^4
\end{array}
\right)\)

\noindent\(10_{140}\text{ : }\left(
\begin{array}{ccc}
 1-m+m^2 &   & -2 m^2 \\
 0 &   & 1-m+m^2
\end{array}
\right)\)

\noindent\(10_{142}\text{ : }\left(
\begin{array}{ccc}
 -1+m-m^2 &   & 1+m-m^2 \\
 0 &   & -2+m+m^2+m^3-2 m^4
\end{array}
\right)\)

\noindent\(10_{144}\text{ : }\left(
\begin{array}{ccc}
 -1+m-m^2 &   & 2 m \\
 0 &   & -3+7 m-3 m^2
\end{array}
\right)\)

\noindent\(10_{147}\text{ : }\left(
\begin{array}{ccc}
 1-2 m &   & -1+4 m-3 m^2 \\
 0 &   & 2-3 m+3 m^2-m^3
\end{array}
\right)\)

\noindent\(10_{155}\text{ : }\left(
\begin{array}{ccc}
 -1+2 m-m^2+m^3 &   & 0 \\
 0 &   & -1+m-2 m^2+m^3
\end{array}
\right)\)

\noindent\(10_{164}\text{ : }\left(
\begin{array}{ccc}
 -1+m-m^2 &   & 3-6 m+4 m^2-m^3 \\
 0 &   & 1-4 m+7 m^2-4 m^3+m^4
\end{array}
\right)\)

\section{Type I colouring matrices}
In this section we list the 216 knots with  colouring matrices that are equivalent to type I matrices and their Alexander polynomials. 

Note that in this case
the number of colourings in a linear Alexander quandle
  $Q=\mathbb{Z}_{n}[t,t^{-1}]\ /\ (t-m)$ is given by $$C_Q(K)=n\times\mbox{gcd}(\mbox{Alex}(m),n)$$ where $\mbox{Alex}(m)$
  is the Alexander polynomial of the knot.
    \vspace{2\baselineskip}
  
\noindent\(3_1\text{ : \  }1-m+m^2\)

\noindent\(4_1\text{: \  }1-3 m+m^2\)

\noindent\(5_1\text{: \  }1-m+m^2-m^3+m^4\)

\noindent\(5_2\text{: \  }2-3 m+2 m^2\)

\noindent\(6_1\text{: \  }2-5 m+2 m^2\)\

\noindent\(6_2\text{: \  }1-3 m+3 m^2-3 m^3+m^4\)

\noindent\(6_3\text{: \  }1-3 m+5 m^2-3 m^3+m^4\)

\noindent\(7_1\text{: \  }1-m+m^2-m^3+m^4-m^5+m^6\)

\noindent\(7_2\text{: \  }3-5 m+3 m^2\)

\noindent\(7_3\text{: \  }2-3 m+3 m^2-3 m^3+2 m^4\)

\noindent\(7_4\text{: \  }4-7 m+4 m^2\)

\noindent\(7_5\text{: \  }2-4 m+5 m^2-4 m^3+2 m^4\)

\noindent\(7_6\text{: \  }1-5 m+7 m^2-5 m^3+m^4\)

\noindent\(7_7\text{: \  }1-5 m+9 m^2-5 m^3+m^4\)

\noindent\(8_1\text{: \  }3-7 m+3 m^2\)

\noindent\(8_2\text{: \  }1-3 m+3 m^2-3 m^3+3 m^4-3 m^5+m^6\)

\noindent\(8_3\text{: \  }4-9 m+4 m^2\)

\noindent\(8_4\text{: \  }2-5 m+5 m^2-5 m^3+2 m^4\)

\noindent\(8_5\text{: \  }1-3 m+4 m^2-5 m^3+4 m^4-3 m^5+m^6\)

\noindent\(8_6\text{: \  }2-6 m+7 m^2-6 m^3+2 m^4\)

\noindent\(8_7\text{: \  }1-3 m+5 m^2-5 m^3+5 m^4-3 m^5+m^6\)

\noindent\(8_8\text{: \  }2-6 m+9 m^2-6 m^3+2 m^4\)

\noindent\(8_9\text{: \  }1-3 m+5 m^2-7 m^3+5 m^4-3 m^5+m^6\)

\noindent\(8_{10}\text{: \  }1-3 m+6 m^2-7 m^3+6 m^4-3 m^5+m^6\)

\noindent\(8_{11}\text{: \  }2-7 m+9 m^2-7 m^3+2 m^4\)

\noindent\(8_{12}\text{: \  }1-7 m+13 m^2-7 m^3+m^4\)

\noindent\(8_{13}\text{: \  }2-7 m+11 m^2-7 m^3+2 m^4\)

\noindent\(8_{14}\text{: \  }2-8 m+11 m^2-8 m^3+2 m^4\)

\noindent\(8_{15}\text{: \  }3-8 m+11 m^2-8 m^3+3 m^4\)

\noindent\(8_{16}\text{: \  }1-4 m+8 m^2-9 m^3+8 m^4-4 m^5+m^6\)

\noindent\(8_{17}\text{: \  }1-4 m+8 m^2-11 m^3+8 m^4-4 m^5+m^6\)

\noindent\(8_{19}\text{: \  }1-m+m^3-m^5+m^6\)

\noindent\(8_{20}\text{: \  }1-2 m+3 m^2-2 m^3+m^4\)

\noindent\(8_{21}\text{: \  }1-4 m+5 m^2-4 m^3+m^4\)

\noindent\(9_1\text{: \  }1-m+m^2-m^3+m^4-m^5+m^6-m^7+m^8\)

\noindent\(9_2\text{: \  }4-7 m+4 m^2\)

\noindent\(9_3\text{: \  }2-3 m+3 m^2-3 m^3+3 m^4-3 m^5+2 m^6\)

\noindent\(9_4\text{: \  }3-5 m+5 m^2-5 m^3+3 m^4\)

\noindent\(9_5\text{: \  }6-11 m+6 m^2\)

\noindent\(9_6\text{: \  }2-4 m+5 m^2-5 m^3+5 m^4-4 m^5+2 m^6\)

\noindent\(9_7\text{: \  }3-7 m+9 m^2-7 m^3+3 m^4\)

\noindent\(9_8\text{: \  }2-8 m+11 m^2-8 m^3+2 m^4\)

\noindent\(9_9\text{: \  }2-4 m+6 m^2-7 m^3+6 m^4-4 m^5+2 m^6\)

\noindent\(9_{10}\text{: \  }4-8 m+9 m^2-8 m^3+4 m^4\)

\noindent\(9_{11}\text{: \  }1-5 m+7 m^2-7 m^3+7 m^4-5 m^5+m^6\)

\noindent\(9_{12}\text{: \  }2-9 m+13 m^2-9 m^3+2 m^4\)

\noindent\(9_{13}\text{: \  }4-9 m+11 m^2-9 m^3+4 m^4\)

\noindent\(9_{14}\text{: \  }2-9 m+15 m^2-9 m^3+2 m^4\)

\noindent\(9_{15}\text{: \  }2-10 m+15 m^2-10 m^3+2 m^4\)

\noindent\(9_{16}\text{: \  }2-5 m+8 m^2-9 m^3+8 m^4-5 m^5+2 m^6\)

\noindent\(9_{17}\text{: \  }1-5 m+9 m^2-9 m^3+9 m^4-5 m^5+m^6\)

\noindent\(9_{18}\text{: \  }4-10 m+13 m^2-10 m^3+4 m^4\)

\noindent\(9_{19}\text{: \  }2-10 m+17 m^2-10 m^3+2 m^4\)

\noindent\(9_{20}\text{: \  }1-5 m+9 m^2-11 m^3+9 m^4-5 m^5+m^6\)

\noindent\(9_{21}\text{: \  }2-11 m+17 m^2-11 m^3+2 m^4\)

\noindent\(9_{22}\text{: \  }1-5 m+10 m^2-11 m^3+10 m^4-5 m^5+m^6\)

\noindent\(9_{23}\text{: \  }4-11 m+15 m^2-11 m^3+4 m^4\)

\noindent\(9_{24}\text{: \  }1-5 m+10 m^2-13 m^3+10 m^4-5 m^5+m^6\)

\noindent\(9_{25}\text{: \  }3-12 m+17 m^2-12 m^3+3 m^4\)

\noindent\(9_{26}\text{: \  }1-5 m+11 m^2-13 m^3+11 m^4-5 m^5+m^6\)

\noindent\(9_{27}\text{: \  }1-5 m+11 m^2-15 m^3+11 m^4-5 m^5+m^6\)

\noindent\(9_{28}\text{: \  }1-5 m+12 m^2-15 m^3+12 m^4-5 m^5+m^6\)

\noindent\(9_{29}\text{: \  }1-5 m+12 m^2-15 m^3+12 m^4-5 m^5+m^6\)

\noindent\(9_{30}\text{: \  }1-5 m+12 m^2-17 m^3+12 m^4-5 m^5+m^6\)

\noindent\(9_{31}\text{: \  }1-5 m+13 m^2-17 m^3+13 m^4-5 m^5+m^6\)

\noindent\(9_{32}\text{: \  }1-6 m+14 m^2-17 m^3+14 m^4-6 m^5+m^6\)

\noindent\(9_{33}\text{: \  }1-6 m+14 m^2-19 m^3+14 m^4-6 m^5+m^6\)

\noindent\(9_{34}\text{: \  }1-6 m+16 m^2-23 m^3+16 m^4-6 m^5+m^6\)

\noindent\(9_{36}\text{: \  }1-5 m+8 m^2-9 m^3+8 m^4-5 m^5+m^6\)

\noindent\(9_{39}\text{: \  }3-14 m+21 m^2-14 m^3+3 m^4\)

\noindent\(9_{42}\text{: \  }1-2 m+m^2-2 m^3+m^4\)

\noindent\(9_{43}\text{: \  }1-3 m+2 m^2-m^3+2 m^4-3 m^5+m^6\)

\noindent\(9_{44}\text{: \  }1-4 m+7 m^2-4 m^3+m^4\)

\noindent\(9_{45}\text{: \  }1-6 m+9 m^2-6 m^3+m^4\)

\noindent\(10_1\text{: \  }4-9 m+4 m^2\)

\noindent\(10_2\text{: \  }1-3 m+3 m^2-3 m^3+3 m^4-3 m^5+3 m^6-3 m^7+m^8\)

\noindent\(10_3\text{: \  }6-13 m+6 m^2\)

\noindent\(10_4\text{: \  }3-7 m+7 m^2-7 m^3+3 m^4\)

\noindent\(10_5\text{: \  }1-3 m+5 m^2-5 m^3+5 m^4-5 m^5+5 m^6-3 m^7+m^8\)

\noindent\(10_6\text{: \  }2-6 m+7 m^2-7 m^3+7 m^4-6 m^5+2 m^6\)

\noindent\(10_7\text{: \  }3-11 m+15 m^2-11 m^3+3 m^4\)

\noindent\(10_8\text{: \  }2-5 m+5 m^2-5 m^3+5 m^4-5 m^5+2 m^6\)

\noindent\(10_9\text{: \  }1-3 m+5 m^2-7 m^3+7 m^4-7 m^5+5 m^6-3 m^7+m^8\)

\noindent\(10_{10}\text{: \  }3-11 m+17 m^2-11 m^3+3 m^4\)

\noindent\(10_{11}\text{: \  }4-11 m+13 m^2-11 m^3+4 m^4\)

\noindent\(10_{12}\text{: \  }2-6 m+10 m^2-11 m^3+10 m^4-6 m^5+2 m^6\)

\noindent\(10_{13}\text{: \  }2-13 m+23 m^2-13 m^3+2 m^4\)

\noindent\(10_{14}\text{: \  }2-8 m+12 m^2-13 m^3+12 m^4-8 m^5+2 m^6\)

\noindent\(10_{15}\text{: \  }2-6 m+9 m^2-9 m^3+9 m^4-6 m^5+2 m^6\)

\noindent\(10_{16}\text{: \  }4-12 m+15 m^2-12 m^3+4 m^4\)

\noindent\(10_{17}\text{: \  }1-3 m+5 m^2-7 m^3+9 m^4-7 m^5+5 m^6-3 m^7+m^8\)

\noindent\(10_{18}\text{: \  }4-14 m+19 m^2-14 m^3+4 m^4\)

\noindent\(10_{19}\text{: \  }2-7 m+11 m^2-11 m^3+11 m^4-7 m^5+2 m^6\)

\noindent\(10_{20}\text{: \  }3-9 m+11 m^2-9 m^3+3 m^4\)

\noindent\(10_{21}\text{: \  }2-7 m+9 m^2-9 m^3+9 m^4-7 m^5+2 m^6\)

\noindent\(10_{22}\text{: \  }2-6 m+10 m^2-13 m^3+10 m^4-6 m^5+2 m^6\)

\noindent\(10_{23}\text{: \  }2-7 m+13 m^2-15 m^3+13 m^4-7 m^5+2 m^6\)

\noindent\(10_{24}\text{: \  }4-14 m+19 m^2-14 m^3+4 m^4\)

\noindent\(10_{25}\text{: \  }2-8 m+14 m^2-17 m^3+14 m^4-8 m^5+2 m^6\)

\noindent\(10_{26}\text{: \  }2-7 m+13 m^2-17 m^3+13 m^4-7 m^5+2 m^6\)

\noindent\(10_{27}\text{: \  }2-8 m+16 m^2-19 m^3+16 m^4-8 m^5+2 m^6\)

\noindent\(10_{28}\text{: \  }4-13 m+19 m^2-13 m^3+4 m^4\)

\noindent\(10_{29}\text{: \  }1-7 m+15 m^2-17 m^3+15 m^4-7 m^5+m^6\)

\noindent\(10_{30}\text{: \  }4-17 m+25 m^2-17 m^3+4 m^4\)

\noindent\(10_{31}\text{: \  }4-14 m+21 m^2-14 m^3+4 m^4\)

\noindent\(10_{32}\text{: \  }2-8 m+15 m^2-19 m^3+15 m^4-8 m^5+2 m^6\)

\noindent\(10_{33}\text{: \  }4-16 m+25 m^2-16 m^3+4 m^4\)

\noindent\(10_{34}\text{: \  }3-9 m+13 m^2-9 m^3+3 m^4\)

\noindent\(10_{35}\text{: \  }2-12 m+21 m^2-12 m^3+2 m^4\)

\noindent\(10_{36}\text{: \  }3-13 m+19 m^2-13 m^3+3 m^4\)

\noindent\(10_{37}\text{: \  }4-13 m+19 m^2-13 m^3+4 m^4\)

\noindent\(10_{38}\text{: \  }4-15 m+21 m^2-15 m^3+4 m^4\)

\noindent\(10_{39}\text{: \  }2-8 m+13 m^2-15 m^3+13 m^4-8 m^5+2 m^6\)

\noindent\(10_{40}\text{: \  }2-8 m+17 m^2-21 m^3+17 m^4-8 m^5+2 m^6\)

\noindent\(10_{41}\text{: \  }1-7 m+17 m^2-21 m^3+17 m^4-7 m^5+m^6\)

\noindent\(10_{42}\text{: \  }1-7 m+19 m^2-27 m^3+19 m^4-7 m^5+m^6\)

\noindent\(10_{43}\text{: \  }1-7 m+17 m^2-23 m^3+17 m^4-7 m^5+m^6\)

\noindent\(10_{44}\text{: \  }1-7 m+19 m^2-25 m^3+19 m^4-7 m^5+m^6\)

\noindent\(10_{45}\text{: \  }1-7 m+21 m^2-31 m^3+21 m^4-7 m^5+m^6\)

\noindent\(10_{46}\text{: \  }1-3 m+4 m^2-5 m^3+5 m^4-5 m^5+4 m^6-3 m^7+m^8\)

\noindent\(10_{47}\text{: \  }1-3 m+6 m^2-7 m^3+7 m^4-7 m^5+6 m^6-3 m^7+m^8\)

\noindent\(10_{48}\text{: \  }1-3 m+6 m^2-9 m^3+11 m^4-9 m^5+6 m^6-3 m^7+m^8\)

\noindent\(10_{49}\text{: \  }3-8 m+12 m^2-13 m^3+12 m^4-8 m^5+3 m^6\)

\noindent\(10_{50}\text{: \  }2-7 m+11 m^2-13 m^3+11 m^4-7 m^5+2 m^6\)

\noindent\(10_{51}\text{: \  }2-7 m+15 m^2-19 m^3+15 m^4-7 m^5+2 m^6\)

\noindent\(10_{52}\text{: \  }2-7 m+13 m^2-15 m^3+13 m^4-7 m^5+2 m^6\)

\noindent\(10_{53}\text{: \  }6-18 m+25 m^2-18 m^3+6 m^4\)

\noindent\(10_{54}\text{: \  }2-6 m+10 m^2-11 m^3+10 m^4-6 m^5+2 m^6\)

\noindent\(10_{55}\text{: \  }5-15 m+21 m^2-15 m^3+5 m^4\)

\noindent\(10_{56}\text{: \  }2-8 m+14 m^2-17 m^3+14 m^4-8 m^5+2 m^6\)

\noindent\(10_{57}\text{: \  }2-8 m+18 m^2-23 m^3+18 m^4-8 m^5+2 m^6\)

\noindent\(10_{58}\text{: \  }3-16 m+27 m^2-16 m^3+3 m^4\)

\noindent\(10_{59}\text{: \  }1-7 m+18 m^2-23 m^3+18 m^4-7 m^5+m^6\)

\noindent\(10_{60}\text{: \  }1-7 m+20 m^2-29 m^3+20 m^4-7 m^5+m^6\)

\noindent\(10_{62}\text{: \  }1-3 m+6 m^2-8 m^3+9 m^4-8 m^5+6 m^6-3 m^7+m^8\)

\noindent\(10_{64}\text{: \  }1-3 m+6 m^2-10 m^3+11 m^4-10 m^5+6 m^6-3 m^7+m^8\)

\noindent\(10_{66}\text{: \  }3-9 m+16 m^2-19 m^3+16 m^4-9 m^5+3 m^6\)

\noindent\(10_{67}\text{: \  }4-16 m+23 m^2-16 m^3+4 m^4\)

\noindent\(10_{68}\text{: \  }4-14 m+21 m^2-14 m^3+4 m^4\)

\noindent\(10_{70}\text{: \  }1-7 m+16 m^2-19 m^3+16 m^4-7 m^5+m^6\)

\noindent\(10_{71}\text{: \  }1-7 m+18 m^2-25 m^3+18 m^4-7 m^5+m^6\)

\noindent\(10_{72}\text{: \  }2-9 m+16 m^2-19 m^3+16 m^4-9 m^5+2 m^6\)

\noindent\(10_{73}\text{: \  }1-7 m+20 m^2-27 m^3+20 m^4-7 m^5+m^6\)

\noindent\(10_{76}\text{: \  }2-7 m+12 m^2-15 m^3+12 m^4-7 m^5+2 m^6\)

\noindent\(10_{77}\text{: \  }2-7 m+14 m^2-17 m^3+14 m^4-7 m^5+2 m^6\)

\noindent\(10_{78}\text{: \  }1-7 m+16 m^2-21 m^3+16 m^4-7 m^5+m^6\)

\noindent\(10_{79}\text{: \  }1-3 m+7 m^2-12 m^3+15 m^4-12 m^5+7 m^6-3 m^7+m^8\)

\noindent\(10_{80}\text{: \  }3-9 m+15 m^2-17 m^3+15 m^4-9 m^5+3 m^6\)

\noindent\(10_{81}\text{: \  }1-8 m+20 m^2-27 m^3+20 m^4-8 m^5+m^6\)

\noindent\(10_{82}\text{: \  }1-4 m+8 m^2-12 m^3+13 m^4-12 m^5+8 m^6-4 m^7+m^8\)

\noindent\(10_{83}\text{: \  }2-9 m+19 m^2-23 m^3+19 m^4-9 m^5+2 m^6\)

\noindent\(10_{84}\text{: \  }2-9 m+20 m^2-25 m^3+20 m^4-9 m^5+2 m^6\)

\noindent\(10_{85}\text{: \  }1-4 m+8 m^2-10 m^3+11 m^4-10 m^5+8 m^6-4 m^7+m^8\)

\noindent\(10_{86}\text{: \  }2-9 m+19 m^2-25 m^3+19 m^4-9 m^5+2 m^6\)

\noindent\(10_{87}\text{: \  }2-9 m+18 m^2-23 m^3+18 m^4-9 m^5+2 m^6\)

\noindent\(10_{88}\text{: \  }1-8 m+24 m^2-35 m^3+24 m^4-8 m^5+m^6\)

\noindent\(10_{89}\text{: \  }1-8 m+24 m^2-33 m^3+24 m^4-8 m^5+m^6\)

\noindent\(10_{90}\text{: \  }2-8 m+17 m^2-23 m^3+17 m^4-8 m^5+2 m^6\)

\noindent\(10_{91}\text{: \  }1-4 m+9 m^2-14 m^3+17 m^4-14 m^5+9 m^6-4 m^7+m^8\)

\noindent\(10_{92}\text{: \  }2-10 m+20 m^2-25 m^3+20 m^4-10 m^5+2 m^6\)

\noindent\(10_{93}\text{: \  }2-8 m+15 m^2-17 m^3+15 m^4-8 m^5+2 m^6\)

\noindent\(10_{94}\text{: \  }1-4 m+9 m^2-14 m^3+15 m^4-14 m^5+9 m^6-4 m^7+m^8\)

\noindent\(10_{95}\text{: \  }2-9 m+21 m^2-27 m^3+21 m^4-9 m^5+2 m^6\)

\noindent\(10_{96}\text{: \  }1-7 m+22 m^2-33 m^3+22 m^4-7 m^5+m^6\)

\noindent\(10_{97}\text{: \  }5-22 m+33 m^2-22 m^3+5 m^4\)

\noindent\(10_{100}\text{: \  }1-4 m+9 m^2-12 m^3+13 m^4-12 m^5+9 m^6-4 m^7+m^8\)

\noindent\(10_{102}\text{: \  }2-8 m+16 m^2-21 m^3+16 m^4-8 m^5+2 m^6\)

\noindent\(10_{104}\text{: \  }1-4 m+9 m^2-15 m^3+19 m^4-15 m^5+9 m^6-4 m^7+m^8\)

\noindent\(10_{105}\text{: \  }1-8 m+22 m^2-29 m^3+22 m^4-8 m^5+m^6\)

\noindent\(10_{107}\text{: \  }1-8 m+22 m^2-31 m^3+22 m^4-8 m^5+m^6\)

\noindent\(10_{109}\text{: \  }1-4 m+10 m^2-17 m^3+21 m^4-17 m^5+10 m^6-4 m^7+m^8\)

\noindent\(10_{110}\text{: \  }1-8 m+20 m^2-25 m^3+20 m^4-8 m^5+m^6\)

\noindent\(10_{111}\text{: \  }2-9 m+17 m^2-21 m^3+17 m^4-9 m^5+2 m^6\)

\noindent\(10_{112}\text{: \  }1-5 m+11 m^2-17 m^3+19 m^4-17 m^5+11 m^6-5 m^7+m^8\)

\noindent\(10_{113}\text{: \  }2-11 m+26 m^2-33 m^3+26 m^4-11 m^5+2 m^6\)

\noindent\(10_{114}\text{: \  }2-10 m+21 m^2-27 m^3+21 m^4-10 m^5+2 m^6\)

\noindent\(10_{116}\text{: \  }1-5 m+12 m^2-19 m^3+21 m^4-19 m^5+12 m^6-5 m^7+m^8\)

\noindent\(10_{117}\text{: \  }2-10 m+24 m^2-31 m^3+24 m^4-10 m^5+2 m^6\)

\noindent\(10_{118}\text{: \  }1-5 m+12 m^2-19 m^3+23 m^4-19 m^5+12 m^6-5 m^7+m^8\)

\noindent\(10_{119}\text{: \  }2-10 m+23 m^2-31 m^3+23 m^4-10 m^5+2 m^6\)

\noindent\(10_{120}\text{: \  }8-26 m+37 m^2-26 m^3+8 m^4\)

\noindent\(10_{121}\text{: \  }2-11 m+27 m^2-35 m^3+27 m^4-11 m^5+2 m^6\)

\noindent\(10_{124}\text{: \  }1-m+m^3-m^4+m^5-m^7+m^8\)

\noindent\(10_{125}\text{: \  }1-2 m+2 m^2-m^3+2 m^4-2 m^5+m^6\)

\noindent\(10_{126}\text{: \  }1-2 m+4 m^2-5 m^3+4 m^4-2 m^5+m^6\)

\noindent\(10_{127}\text{: \  }1-4 m+6 m^2-7 m^3+6 m^4-4 m^5+m^6\)

\noindent\(10_{128}\text{: \  }2-3 m+m^2+m^3+m^4-3 m^5+2 m^6\)

\noindent\(10_{129}\text{: \  }2-6 m+9 m^2-6 m^3+2 m^4\)

\noindent\(10_{130}\text{: \  }2-4 m+5 m^2-4 m^3+2 m^4\)

\noindent\(10_{131}\text{: \  }2-8 m+11 m^2-8 m^3+2 m^4\)

\noindent\(10_{132}\text{: \  }1-m+m^2-m^3+m^4\)

\noindent\(10_{133}\text{: \  }1-5 m+7 m^2-5 m^3+m^4\)

\noindent\(10_{134}\text{: \  }2-4 m+4 m^2-3 m^3+4 m^4-4 m^5+2 m^6\)

\noindent\(10_{135}\text{: \  }3-9 m+13 m^2-9 m^3+3 m^4\)

\noindent\(10_{136}\text{: \  }1-4 m+5 m^2-4 m^3+m^4\)

\noindent\(10_{137}\text{: \  }1-6 m+11 m^2-6 m^3+m^4\)

\noindent\(10_{138}\text{: \  }1-5 m+8 m^2-7 m^3+8 m^4-5 m^5+m^6\)

\noindent\(10_{139}\text{: \  } 1-m+2 m^3-3 m^4+2 m^5-m^7+m^8\)

\noindent\(10_{141}\text{: \  }1-3 m+4 m^2-5 m^3+4 m^4-3 m^5+m^6\)

\noindent\(10_{143}\text{: \  }1-3 m+6 m^2-7 m^3+6 m^4-3 m^5+m^6\)

\noindent\(10_{145}\text{: \  }1+m-3m^2+m^3+m^4\) 

\noindent\(10_{146}\text{: \  }2-8 m+13 m^2-8 m^3+2 m^4\)

\noindent\(10_{148}\text{: \  }1-3 m+7 m^2-9 m^3+7 m^4-3 m^5+m^6\)

\noindent\(10_{149}\text{: \  }1-5 m+9 m^2-11 m^3+9 m^4-5 m^5+m^6\)

\noindent\(10_{150}\text{: \  }1-4 m+6 m^2-7 m^3+6 m^4-4 m^5+m^6\)

\noindent\(10_{151}\text{: \  }1-4 m+10 m^2-13 m^3+10 m^4-4 m^5+m^6\)

\noindent\(10_{152}\text{: \  }1-m-m^2+4 m^3-5 m^4+4 m^5-m^6-m^7+m^8\)

\noindent\(10_{153}\text{: \  }1-m-m^2+3 m^3-m^4-m^5+m^6\)

\noindent\(10_{154}\text{: \  }1-4 m^2+7 m^3-4 m^4+m^6\)

\noindent\(10_{156}\text{: \  }1-4 m+8 m^2-9 m^3+8 m^4-4 m^5+m^6\)

\noindent\(10_{158}\text{: \  }1-4 m+10 m^2-15 m^3+10 m^4-4 m^5+m^6\)

\noindent\(10_{159}\text{: \  }1-4 m+9 m^2-11 m^3+9 m^4-4 m^5+m^6\)

\noindent\(10_{161}\text{: \  }1-2 m^2+3 m^3-2 m^4+m^6\)

\noindent\(10_{163}\text{: \  }3-9 m+11 m^2-9 m^3+3 m^4\)

\noindent\(10_{165}\text{: \  }3-11 m+17 m^2-11 m^3+3 m^4\)

\noindent\(10_{166}\text{: \  }2-10 m+15 m^2-10 m^3+2 m^4\)

\newpage
 \nocite{*} 
\bibliography{colourings}
\bibliographystyle{plain}
\end{document}